\documentclass[12pt]{amsart}
\usepackage{amscd, amsfonts, amssymb, mathrsfs}
\usepackage{fullpage}
\usepackage{latexsym}
\usepackage{verbatim}
\usepackage{enumerate}
\usepackage[all,cmtip]{xy}
\usepackage[colorlinks=true]{hyperref}

\newcommand{\iso}{\cong}
\numberwithin{equation}{section}

\newtheorem{thm}[equation]{Theorem}

\newtheorem{lem}[equation]{Lemma}
\newtheorem{cor}[equation]{Corollary}
\newtheorem{prop}[equation]{Proposition}

\theoremstyle{definition}
\newtheorem{defn}[equation]{Definition}
\newtheorem{ex}[equation]{Example}

\theoremstyle{remark}
\newtheorem{rem}[equation]{Remark}
\theoremstyle{remark}
\newtheorem{rems}[equation]{Remarks}

\newcommand{\ad}{{\rm ad}}
\newcommand{\GL}{{\rm GL}}
\newcommand{\SL}{{\rm SL}}
\newcommand{\PGL}{{\rm PGL}}
\newcommand{\frakgl}{\mathfrak {gl}}
\newcommand{\fraksl}{\mathfrak {sl}}
\renewcommand{\Im}{{\rm Im}}
\renewcommand{\char}{{\rm char}}
\newcommand{\frakc}{\mathfrak c}
\newcommand{\frakg}{\mathfrak g}
\newcommand{\frakh}{\mathfrak h}
\newcommand{\frakp}{\mathfrak p}
\newcommand{\frakk}{\mathfrak k}
\newcommand{\frakl}{\mathfrak l}
\newcommand{\frakm}{\mathfrak m}
\newcommand{\frakn}{\mathfrak n}
\newcommand{\fraks}{\mathfrak s}

\newcommand{\frakz}{\mathfrak z}

\newcommand{\NN}{\mathbb N}
\newcommand{\Gm}{\mathbb{G}_m}
\newcommand{\calN}{\mathscr N}
\newcommand{\calC}{\mathscr C}

\newcommand{\twobytwo}[4]{{\left(\begin{array}{ll} #1 & #2 \\ #3 & #4 \\\end{array}\right)}}
\newcommand{\threebythree}[9]{{\left(\begin{array}{lll} #1 & #2 & #3 \\ #4 & #5 & #6 \\ #7 & #8 & #9 \end{array}\right)}}

\DeclareMathOperator{\Lie}{Lie}
\DeclareMathOperator{\Inn}{Inn}
\DeclareMathOperator{\Ad}{Ad}
\DeclareMathOperator{\rad}{rad}

\newcommand{\ovl}{\overline}

\newcommand{\tuple}[1]{{\mathbf {#1}}}


\dedicatory{To the memory of Irina Suprunenko} 

\title[Complete reducibility for Lie subalgebras and semisimplification]
{Complete reducibility for Lie subalgebras and semisimplification} 
 
 \author[M.\ Bate]{Michael Bate}
 \address%[M.\ Bate]
 {Department of Mathematics,
 	University of York,
 	York YO10 5DD,
 	United Kingdom}
 \email{michael.bate@york.ac.uk}

\author[S. B\"ohm]{S\"oren B\"ohm}
\address
{Fakult\"at f\"ur Mathematik,
	Ruhr-Universit\"at Bochum,
	D-44780 Bochum, Germany}
\email{soeren.boehm@rub.de}
 
 \author[B.\ Martin]{Benjamin Martin}
 \address%[B.\ Martin]
 {Department of Mathematics,
 	University of Aberdeen,
 	King's College,
 	Fraser Noble Building,
 	Aberdeen AB24 3UE,
 	United Kingdom}
 \email{b.martin@abdn.ac.uk}

\author[G. R\"ohrle]{Gerhard R\"ohrle}
\address
{Fakult\"at f\"ur Mathematik,
Ruhr-Universit\"at Bochum,
D-44780 Bochum, Germany}
\email{gerhard.roehrle@rub.de}

\author[L. Voggesberger]{Laura Voggesberger}
\address
{Fakult\"at f\"ur Mathematik,
	Ruhr-Universit\"at Bochum,
	D-44780 Bochum, Germany}
\email{laura.voggesberger@rub.de}

\begin{document}

\begin{abstract}
	 Let $G$ be a connected reductive linear algebraic group over a field $k$.  Using ideas from geometric invariant theory, we study the notion of $G$-complete reducibility over $k$ for a Lie subalgebra $\frakh$ of the Lie algebra $\frakg = \Lie(G)$ of $G$ and prove some results when $\frakh$ is solvable or $\char(k)= 0$.  We introduce the concept of a \emph{$k$-semisimplification} $\frakh'$ of $\frakh$; $\frakh'$ is a Lie subalgebra of $\frakg$ associated to $\frakh$ which is $G$-completely reducible over $k$. 
	This is the Lie algebra counterpart of the analogous notion for subgroups studied earlier by the first, third and fourth authors.
	As in the subgroup case, we show that $\frakh'$ is unique up to $\Ad(G(k))$-conjugacy in $\frakg$. Moreover, we prove that the two concepts are compatible: 
	for $H$ a closed subgroup of $G$ and $H'$ a $k$-semisimplification of $H$, the Lie algebra $\Lie(H')$ is a $k$-semisimplification of $\Lie(H)$.
\end{abstract}

\maketitle

\section{Introduction}
\label{sec:intro}

Let $G$ be a connected reductive linear algebraic group over an arbitrary field  $k$. We revisit the notion of \emph{$G$-complete reducibility over $k$} for a 
Lie subalgebra $\frakh$ of the Lie algebra $\frakg:= \Lie(G)$ of $G$ from \cite[Def.\ 5.3]{sphericalcochar}.
A central theme of the present paper is an extension 
to the setting of Lie algebras
of the construction of a 
\emph{$k$-semisimplification} of a subgroup $H$ of $G$ from \cite{BMR:semisimplification}.
For motivation and an overview of the latter and further references, we refer to the introduction of \cite{BMR:semisimplification}.
The concept of the semisimplification of a module for a group or algebra is a well-known construction in representation theory.
Building on this idea, for $H$ a subgroup of $G(k)$, Serre introduced the concept of a ``$G$-analogue'' of semisimplification from representation theory in \cite[\S 3.2.4]{serre2}.
As in the subgroup case, one can think of the idea of a $k$-semisimplification of a subalgebra of $\frakg$ as an analogue of the 
Jordan-H\"older Theorem.

The definition of a \emph{$k$-semisimplification} of a Lie subalgebra of $\frakg$ (Definition~\ref{defn:liess}) for arbitrary $k$ is new and 
generalizes the one for subgroups of $G$ (Definition~\ref{defn:ss}).
We prove that the $k$-semisimplification of a subalgebra $\frakh$ of $\frakg$ is unique up to $\Ad(G(k))$-conjugacy (Theorem~\ref{thm:liemain}). 
Moreover, for $H$ a closed subgroup of $G$ and $H'$ a $k$-semisimplification of $H$, we show that $\Lie(H')$ is a $k$-semisimplification of $\Lie(H)$ (Theorem~\ref{thm:ssLievsGp}).

To prove Theorem~\ref{thm:liemain} we use the theory of $G$-complete reducibility for subalgebras of $\frakg$ and some ideas from geometric invariant theory.  In Theorem~\ref{thm:lienormal} we show that when ${\rm char}(k)$ is large enough, an ideal of a $G$-completely reducible subalgebra $\frakh$ is $G$-completely reducible and the process of $k$-semisimplification behaves well under passing to ideals of $\frakh$.

We also prove some results of independent interest about $G$-complete reducibility for subalgebras, including Proposition~\ref{prop:finite}.  We consider solvable subalgebras in Section~\ref{sec:toral_solvable}.  Theorem~\ref{thm:Lie_Gcr_crit} gives a necessary condition for a subalgebra to be $G$-completely reducible when $\char(k)=p > 0$ is sufficiently large, and yields a characterisation of $G$-completely reducible subalgebras when $\char(k)= 0$.  We discuss some related results of Richardson, who pioneered the application of  geometric invariant theory to the study of $G$-complete reducibility.

The notion of $G$-complete reducibility for subalgebras of $\frakg$ was first defined by McNinch \cite{mcninch} for algebraically closed $k$ and was developed further in \cite[\S 3]{BMRT:relative} by the first, third and fourth authors; the non-algebraically closed case was first studied in \cite[\S 5]{sphericalcochar}.  The approach via geometric invariant theory stems from work of Richardson \cite{rich}, who studied subgroups of $G$ and subalgebras of $\Lie(G)$ --- mainly for algebraically closed $k$ in characteristic 0.  Some of our arguments involve extending constructions from \emph{op.\ cit.}\ to positive characteristic.  We also use more recent techniques from \cite{BMR} and \cite{BMR:semisimplification} including some deep results from the theory of spherical buildings.

The theory of $G$-complete reducibility for Lie subalgebras closely parallels the theory of $G$-complete reducibility for subgroups \cite{BMR}, \cite{BMR:semisimplification}.  Some arguments are easier in the subalgebra case, as we need not worry about non-connected groups.  There are, however, some extra problems for subalgebras, which can be traced back to the failure of certain normalisers to be smooth.  It is to avoid these difficulties that we need assumptions on the characteristic of $k$: see Section~\ref{subsec:smooth}.

\section{Preliminaries}
\label{sec:prelims}

\subsection{Basic notation}
Let $k$ be a field with $\char(k) = p \ge 0$.  Let $\ovl{k}$ and $k_s$ be the algebraic closure and separable closure of $k$, respectively.
Following \cite{borel}, \cite{GIT}, and \cite{cochar}, we regard an affine variety over a field $k$ as a variety $X$ over the algebraic closure $\ovl{k}$ together with a choice of $k$-structure. 
We write $X(k)$ for the set of $k$-points of $X$ and $X(k')$ for the set of $k'$-points of $X$ if $k'/k$ is a field extension; we often write just $X$ for $X(\ovl{k})$.  By a subvariety of $X$ we mean a closed $\ovl{k}$-subvariety of $X$; a $k$-subvariety is a subvariety that is defined over $k$. 

Throughout, $G$ denotes a connected reductive linear algebraic $k$-group.  Algebraic groups and their subgroups are assumed to be smooth (although we consider certain non-smooth subgroup schemes in Section~\ref{subsec:smooth}).
A $k$-defined affine $G$-variety $X$ is an affine variety over $k$ equipped with a $k$-defined morphic action of $G$.
Two important examples of affine $G$-varieties in this paper are the group $G$ itself, with $G$ acting by inner automorphisms, and 
the Lie algebra $\frakg = \Lie(G)$ of $G$, with 
$G$ acting by the adjoint action (that $\frakg$ admits a $k$-structure is \cite[(12.2.3), (4.4.8)]{springer}).
Recall that these two actions are closely related: if $g\in G$, then we have $\Inn(g):G\to G, x\mapsto gxg^{-1}$, and then
$\Ad(g):= d(\Inn(g)):\frakg \to \frakg$ is the differential (see below).
To simplify notation we denote both of these actions by a dot: that is, if $g\in G$, then we let 
$g\cdot x = gxg^{-1}$ for each $x\in G$, and $g\cdot x = \Ad(g)(x)$ for each $x\in \frakg$.  We define $\ad$ to be the usual adjoint action of $\frakg$ on $\frakg$: so $\ad(x)(y)= [x, y]$ for $x, y\in \frakg$.  If necessary we write $\Ad_G$ instead of $\Ad$ and $\ad_\frakg$ instead of $\ad$.
Given any $m\in \NN$, we may extend these actions to actions of $G$ on $X = G^m$ or $\frakg^m$, the $m$-fold Cartesian product of $G$ or $\frakg$, by setting
$$
g\cdot (x_1,\ldots,x_m) := (g\cdot x_1,\ldots, g\cdot x_m)
$$
for each $g\in G$ and $(x_1,\ldots,x_m)\in X$.
We refer to this action informally in both cases as the action \emph{by simultaneous conjugation}.

When $\char(k) = p > 0$, the Lie algebra $\frakg$ is \emph{restricted} with a \emph{$p$-operation} $\frakg \to \frakg, x \mapsto x^{[p]}$, \cite[(4.4.3)]{springer}. 
By a subgroup of $G$ we mean a closed $\ovl{k}$-subgroup and by a $k$-subgroup we mean a subgroup that is defined over $k$. 
For a subgroup $H$ of $G$ we write $H^0$ for the identity component of $H$ and $\Lie(H)$ for its Lie algebra (which is a subalgebra of $\frakg$).  Subalgebras of $\frakg$ of the form $\Lie(H)$ as above are called \emph{algebraic}.

If $H$ is $k$-defined, then $\Lie(H)$ also admits a $k$-structure, \cite[(12.2.3), (4.4.8)]{springer}.  We write $\frakg(k')$ for the set of $k'$-points of $\frakg$ if $k'/k$ is a field extension, and just $\frakg$ for $\frakg(\ovl{k})$.  By a $k$-subalgebra of $\frakg$ we mean a Lie subalgebra $\frakh$ over $\ovl{k}$ which is defined over $k$.

If $f\colon G_1\to G_2$ is a homomorphism of algebraic groups then we denote the induced map from $\Lie(G_1)$ to $\Lie(G_2)$ by $df$.

\subsection{Cocharacters and parabolic subgroups} 
We recall some basic notation and facts 
concerning cocharacters and parabolic subgroups in connected reductive groups $G$ from \cite{GIT}, \cite{rich} and \cite{springer}.
Recall that a cocharacter of $G$ is a homomorphism of algebraic groups $\lambda:\Gm\to G$, where $\Gm$ is the multiplicative group.
We define $Y_k(G)$ to be the set of $k$-defined cocharacters of $G$ and $Y(G) := Y_{\ovl{k}}(G)$ to be the set of all cocharacters of $G$.
Given $\lambda\in Y(G)$, we define 
$$
P_\lambda= \{g\in G\,|\,\lim_{a\to 0} \lambda(a)g\lambda(a)^{-1}\ \mbox{exists}\}
$$
and $L_\lambda= C_G({\rm Im}(\lambda))\subseteq P_\lambda$ (for the definition of a limit, see \cite[\S 3.2.13]{springer}). 
We have $P_\lambda= L_\lambda= G$ if and only if $\Im(\lambda)$ is contained in the centre of $G$.
Any parabolic subgroup (resp., any Levi subgroup of a parabolic subgroup of $G$) is of the form $P_\lambda$ (resp., $L_\lambda$) \cite[Prop.~8.4.5]{springer}. 
In fact, this also holds for $k$-defined parabolic and Levi subgroups of $G$, by the following, which is \cite[Lem.~15.1.2(ii)]{springer}:

\begin{lem}\label{lem:P,L is Pl, Ll}
Every pair $(P,L)$ consisting of a parabolic $k$-subgroup $P$ of $G$ and a Levi $k$-subgroup $L$ of $P$ is of the form $(P,L)= (P_\lambda, L_\lambda)$ for some $\lambda\in Y_k(G)$, and vice versa.
\end{lem}

For ease of reference, we record without proof some basic facts about these subgroups. The unipotent radical of a parabolic subgroup $P$ of $G$ is denoted by $R_u(P)$.
The following is \cite[2.3]{rich}.

\begin{lem}\label{lem:k-definedRuP}
If $P$ is a $k$-defined parabolic subgroup then $R_u(P)$ is $k$-defined.
\end{lem}

The following is \cite[Lem.~ 2.5 (i)+(iii)]{GIT}. See also \cite[16.1.1]{springer}.

\begin{lem}
\label{lem:parprops}
Let $P$ be a parabolic subgroup of $G$ and $L$ a Levi subgroup of $P$.
\begin{itemize}
\item[(i)] We have $P\iso L\ltimes R_u(P)$, and this is a $k$-isomorphism if $P$ and $L$ are $k$-defined. 
\item[(ii)] Let $T$ be a maximal torus of $P$. Then there is a unique Levi subgroup $L$ of $P$ such that $T\subseteq L$. If $P$ and $T$ are $k$-defined then $L$ is $k$-defined.
\item[(iii)] Any two Levi $k$-subgroups of a parabolic $k$-subgroup $P$ are $R_u(P)(k)$-conjugate. 
\end{itemize}
\end{lem}

We denote the canonical projection from $P$ to $L$ by $c_L$; this is $k$-defined if $P$ and $L$ are. 
Given any cocharacter $\lambda$ of $G$ such that $P= P_\lambda$ and $L= L_\lambda$ then we often write $c_\lambda$ instead of $c_L$. We have $c_\lambda(g)= \lim_{a\to 0} \lambda(a)g\lambda(a)^{-1}$ for $g\in P_\lambda$; the kernel of $c_\lambda$ is the unipotent radical $R_u(P_\lambda)$ and the set of fixed points of $c_\lambda$ is $L_\lambda$.

Let $M$ be a connected reductive subgroup of $G$ and let $\lambda\in Y(M)$.  We can perform the above constructions in $M$ and contain subgroups $P_\lambda(M)$ and $L_\lambda(M)$.  As above, we write $P_\lambda$ for $P_\lambda(G)$ and $L_\lambda$ for $L_\lambda(G)$.  It follows easily from the definitions and results above that $P_\lambda(M)= P_\lambda\cap M$, $L_\lambda(M)= L_\lambda\cap M$ and $R_u(P_\lambda(M))= R_u(P_\lambda)\cap M$.

All of this transfers over to the Lie algebra $\frakg$ through the adjoint action: since
$G$ acts on $\frakg$ by the adjoint action, so does any subgroup of $G$ and, in particular, the image of any cocharacter of $G$.
For $\lambda\in Y(G)$, we define and use throughout $\frakp_\lambda:=\Lie(P_\lambda)$ and
$\frakl_\lambda:=\Lie(L_\lambda)$.
We require some standard facts concerning
Lie algebras of parabolic and Levi subgroups of $G$
(cf.~\cite[\S 2.1]{rich}).

\begin{lem}
	\label{lem:liealgebrasofRpars}
	Let $x\in\frakg$. Then with the notation from above we have
	\begin{enumerate}[{\rm(i)}]
		\item $x\in\frakp_\lambda$ if and only if $\underset{a\to 0}{\lim}\,
		\lambda(a)\cdot x$ exists;
		\item $x\in\frakl_\lambda$ if and only if $\underset{a\to 0}{\lim}\,
		\lambda(a)\cdot x = x$ if and only if the image of $\lambda$ centralizes $x$;
		\item $x\in\Lie(R_u(P_\lambda))$ if and only if $\underset{a\to 0}{\lim}\,
		\lambda(a)\cdot x$ exists and equals $0$.
	\end{enumerate}
\end{lem}

The map $c_\lambda:=c_{\frakl_\lambda} : \frakp_\lambda \to \frakl_\lambda$ given
by $x \mapsto \underset{a\to 0}{\lim}\, \lambda(a)\cdot x$, where the action is given by adjoint action, coincides with
the usual projection of $\frakp_\lambda$ onto $\frakl_\lambda$.
 The kernel of $c_\lambda$ is $\Lie(R_u(P_\lambda))$, and the set of fixed points of $c_\lambda$ is $\frakl_\lambda$. 

Let $m\in\NN$ and recall that $G$ acts on $\frakg^m$ by simultaneous conjugation. 
Given $\lambda\in Y(G)$, we have a map $\frakp_\lambda^m\to \frakl_\lambda^m$ given by ${\mathbf x}\mapsto \lim_{a\to 0} \lambda(a)\cdot {\mathbf x}$; we abuse notation slightly and also call this map $c_\lambda$. For any ${\mathbf x}\in \frakp_\lambda^m$, there exists a Levi $k$-subgroup $L$ of $P_\lambda$ with ${\mathbf x}\in \frakl^{m}$ if and only if $c_\lambda({\mathbf x})= u\cdot {\mathbf x}$ for some $u\in R_u(P_\lambda)(k)$, by \cite[Prop.~2.11]{cochar}.

\subsection{Smoothness of centralisers and normalisers}
\label{subsec:smooth}

We are concerned with linear algebraic groups $G$ over $k$: these are  affine group schemes of finite type over $k$ (assumed to be smooth).  Note that smoothness is automatic if $\char(k)= 0$, by a result of Cartier \cite[II, \S 6, 1.1]{DG70}.  We do need to consider certain non-smooth subgroup schemes of $G$, as the lack of smoothness has implications for $G$-complete reducibility.  If $\frakh$ is a subalgebra of $\frakg$ then we define $\calN_G(\frakh)$ to be the \emph{scheme-theoretic} normaliser of $\frakh$ in $G$, and $\calC_G(\frakh)$ to be the \emph{scheme-theoretic} centraliser of $\frakh$ in $G$ (see \cite[I.2.12]{Jantzen2} or \cite[A.1.9]{CGP}).  We define $N_G(\frakh)$ and $C_G(\frakh)$ to be the setwise normaliser and centraliser of $\frakh$, respectively.  Note that $\calN_G(\frakh)$ and $\calC_G(\frakh)$ are $k$-defined, but $N_G(\frakh)$ and $C_G(\frakh)$ need not be $k$-defined.  The group scheme $\calN_G(\frakh)$ (resp., $\calC_G(\frakh)$) is smooth if and only if $\calN_G(\frakh)= N_G(\frakh)$ (resp., $\calC_G(\frakh)= C_G(\frakh)$), and in this case $N_G(\frakh)$ (resp., $C_G(\frakh)$) is $k$-defined.

There is a criterion for smoothness in terms of the Lie algebra.  Define
$$ \frakn_\frakg(\frakh) := \{x\in \frakg \mid [x, \frakh]\subseteq \frakh\} $$
and
$$ \frakc_\frakg(\frakh) := \{x\in \frakg \mid [x, y]= 0\ \mbox{for all $y\in \frakh$}\}. $$
It can be shown that $\calN_G(\frakh)$ is smooth if and only if $\frakn_\frakg(\frakh) = \Lie(N_G(\frakh))$ and $\calC_G(\frakh)$ is smooth if and only if $\frakc_\frakg(\frakh) = \Lie(C_G(\frakh))$: e.g.,~see \cite[Prop.~II.6.7]{borel}.  For further discussion, see \cite[\S 3]{herpel} and \cite{HeSt}.
For exact conditions for $\calC_G(\frakh)$ to be smooth, see 
\cite[Thm.~1.2]{BMRT} and \cite{herpel};
for explicit lower bounds on $\char(k)$ to ensure that 
$\calN_G(\frakh)$ is smooth, see \cite[Thm.~A]{HeSt}.

If $H$ is a $k$-subgroup of $G$ then we define $\calN_G(H)$, $\calC_G(H)$, $N_G(H)$, $C_G(H)$, $\frakn_\frakg(H)$ and $\frakc_\frakg(H)$ in the obvious way, and analogous results to those above hold: so, for instance, $\frakc_\frakg(H)= \{x\in \frakg\,|\,h\cdot x= x\ \mbox{for all $h\in H$}\}$, and $\calC_G(H)$ is smooth if and only if $\calC_G(H)= C_G(H)$ if and only if $\frakc_\frakg(H) = \Lie(C_G(H))$.  We recall a useful fact \cite[III.9.2 Cor.]{borel}: if $S$ is a torus in a linear algebraic group $M$ then $\calC_M(S)$ is smooth.

\begin{defn}
 \begin{itemize}
  \item[(a)] Let $\Gamma$ be a subgroup of $G(k)$ consisting of unipotent elements.  We call $\Gamma$ \emph{plongeable} if there is a parabolic subgroup $P$ of $G$ such that $\Gamma\subseteq R_u(P)(k)$ (see \cite[Sec.\ 1]{tits3}).
  \item[(b)] Let $\frakh$ be a $k$-subalgebra of $\frakg$ consisting of nilpotent elements.  We call $\frakh$ \emph{plongeable} if there is a parabolic subgroup $P$ of $G$ such that $\frakh\subseteq \Lie(R_u(P))$.
 \end{itemize}
\end{defn}

\begin{rem}
\label{rem:plong_subgp}
 Let $M$ be a connected reductive $k$-subgroup of $G$ and let $\Gamma$ be a subgroup of $M(k)$ consisting of unipotent elements.  If $\Gamma$ is plongeable in $M$ then there exists $\lambda\in Y_k(M)$ such that $\Gamma\subseteq R_u(P_\lambda(M))(k)$.  Then $\Gamma\subseteq R_u(P_\lambda(G))(k)$, so $\Gamma$ is plongeable in $G$.  Likewise, if $\frakh$ is a $k$-subalgebra of $\Lie(M)$ consisting of nilpotent elements and $\frakh$ is plongeable in $M$ then $\frakh$ is plongeable in $G$.
\end{rem}

\begin{defn}
	\label{def:fabulous}
 Define $\char(k) = p > 0$ to be \emph{fabulous} for $G$  if the following hold:
 
 (a) Centralisers of $k$-subgroups of $G$ and $k$-subalgebras of $\frakg$ are smooth.
 
 (b) Normalisers of $k$-subalgebras of $\frakg$ are smooth.
 
 (c) If $\frakh$ is a subalgebra of $\frakg$ consisting of nilpotent elements then $\frakh\subseteq \Lie(R_u(P))$ for some $k$-parabolic subgroup $P$ of $G$.
\end{defn}

\begin{rems}
\label{rems:fabulous}
         (i) Suppose $k= \ovl{k}$ and let $\frakh$ be a subalgebra of $\frakg$.  Then $\frakh$ is plongeable if and only if $\frakh\subseteq \Lie(B)$ for some Borel subgroup $B$ of $G$, and this is the case if and only if $\frakh\subseteq \Lie(U)$ for some maximal unipotent subgroup $U$ of $G$.
        
	(ii) If $\char(k) = 0$, then all three conditions in Definition~\ref{def:fabulous} are satisfied, see \cite[II, §6, 1.1]{DG70}, \cite[Ch.~VIII, \S 10, Cor.~2]{Bou05}.
		
	(iii) Suppose $k= \ovl{k}$.  If $\frakh$ consists of nilpotent elements and $\dim(\frakh)= 1$ then $\frakh\subseteq \Lie(B)$ for some Borel subgroup $B$ of $G$: this follows from \cite[IV.14.25 Prop.]{borel}.  This can fail if $k\neq \ovl{k}$.  For example, let $G= \PGL_2(k)$.  For $A\in \GL_2(k)$, we denote the image of $A$ in $G$ by $\ovl{A}$, and likewise for elements of $\frakgl_2(k)$.  If $\char(k)= 2$ and we choose $a\in k^{\frac{1}{2}}$ such that $a\not\in k$ then the subspace of $\Lie(\PGL_2)$ spanned by $\ovl{\twobytwo{0}{1}{a^2}{0}}$ consists of nilpotent elements but is not contained in $\Lie(B)$ for any Borel subgroup $B$ of $\PGL_2$ (compare \cite[Rem.\ 5.10]{GIT}).  This is the Lie algebra  counterpart of the phenomenon of ``non-plongeabilit\'e'' for unipotent elements of the group, cf.~\cite{tits2}.
	
	(iv) If condition (a) of Definition~\ref{def:fabulous} holds then $Z(G)= \calC_G(G)$ is smooth, so $\frakz(\frakg)= \Lie(Z(G))$; in particular, $\frakz(\frakg)= 0$ in this case if $G$ is semisimple.
	
	(v) It is immediate that if condition (a) of Definition~\ref{def:fabulous} holds for $G$ regarded as a $\ovl{k}$-group then it also holds for $G$ regarded as a $k$-group.
\end{rems}

\begin{ex}
\label{ex:PGL_2}
 Suppose $k$ is algebraically closed of characteristic $p= 2$  and let $G= \PGL_2(k)$.  Let $\frakh$ be the abelian subalgebra of $\frakg$ spanned by $x_1:= \ovl{\twobytwo{0}{1}{0}{0}}$ and $x_2:= \ovl{\twobytwo{0}{0}{1}{0}}$.  It is easy to check that $\frakh$ is not contained in $\Lie(B)$ for any Borel subgroup $B$ of $G$.  Similar examples are produced in \cite{lmt} for any $G$ such that $p$ is a torsion prime for $G$.
 
 Now let $\frakm$ be the subalgebra of $\frakg$ spanned by $x_1$.   Then $C_G(\frakm)(k)$ is (the image of) the group of upper unitriangular matrices and $N_G(\frakm)(k)$ is (the image of) the group of upper triangular matrices, but $\frakc_\frakg(\frakm)= \frakh$ and $\frakn_\frakg(\frakm)= \frakg$.  Hence neither $C_G(\frakm)$ nor $N_G(\frakm)$ is smooth.  A similar calculation shows that neither $C_G(\frakh)$ nor $N_G(\frakh)$ is smooth.
\end{ex}

Suppose $k= \ovl{k}$.  Owing to \cite[Thm.~1.1]{herpel}, centralisers of subgroups of $G$ and centralisers of subalgebras of $\frakg$ are smooth if and only if $\char(k)$ is $0$ or a ``pretty good prime'' $p$ for $G$,  (\cite[Def.~2.11]{herpel}). The latter implies that $p$ is not a torsion prime for the root system of $G$.  It then follows from \cite[Thm.\ 2.2 and Rems.(a)]{lmt} that condition (c) of Definition~\ref{def:fabulous} holds.  Hence condition (a) implies condition (c) in this instance (see also \cite[Thm.~9.2]{HeSt} and \cite[Cor.\ 1.3]{premetstewart}).  We now extend this result to arbitrary fields (Proposition~\ref{prop:fab_crit}), using results from \cite{tits3} on plongeable unipotent subgroups.  First we show that condition (a) is inherited by Levi subgroups and certain quotients of $G$.

\begin{lem}
\label{lem:pretty_good_descent}
 Let $G$ be connected.  Suppose the centraliser of every $k$-subgroup of $G$ and every $k$-subalgebra of $\frakg$ is smooth.
 \begin{itemize}
  \item[(a)] Let $S$ be a $k$-torus of $G$ and let $L= C_G(S)$.  Then the centraliser in $L$ of every $k$-subgroup of $L$ and every $k$-subalgebra of $\Lie(L)$ is smooth.
  \item[(b)] Let $S$ be a $k$-subgroup of $Z(G)^0$.  Then the centraliser in $L$ of every $k$-subgroup of $G/S$ and every $k$-subalgebra of $\Lie(G/S)$ is smooth.
 \end{itemize}
\end{lem}

\begin{proof}
 (a) Let $H$ be a $k$-subgroup of $L$.  Set $M= C_G(H)$ and $\frakm= \Lie(M)$.  Then $S\subseteq M$ and $\frakm= \frakc_\frakg(H)$ by assumption.  We have $C_L(H)= C_M(S)$ and $\frakc_\frakl(H)= \frakc_\frakm(S)$.  But centralisers of tori are smooth, so
 $$ \Lie(C_L(H))= \Lie(C_M(S))= \frakc_\frakm(S)= \frakc_\frakl(H). $$
 The proof for centralisers of $k$-subalgebras is similar.
 
 (b) Let $G_1= G/S$ and let $\pi\colon G\to G_1$ denote the canonical projection.  Let $H_1$ be a $k$-subgroup of $G_1$ and let $H= \pi^{-1}(H_1)$; note that $H$ is $k$-defined as $\pi$ is a smooth map.  Let $M_1= C_{G_1}(H_1)^0$ and let $M= \pi^{-1}(M_1)^0$.  We claim that $M= C_G(H)^0$.  Clearly $M\supseteq C_G(H)^0$.  For the reverse inclusion, if $h\in H$ then the map $M\to [G,G]$, $m\mapsto [m,h]$ has image in $[G,G]\cap S$.  But $M$ is connected and $[G,G]\cap S$ is finite, so $[m,h]= 1$ for all $m\in M$.  Hence $M\subseteq C_G(H)^0$, as required.  We deduce that $\dim(M)= \dim(M_1)+ \dim(S)$.  By a similar argument, $\frakc_\frakg(H)= d\pi^{-1}(\frakc_{\frakg_1}(H_1))$ and $\dim(\frakc_\frakg(H))= \dim(\frakc_{\frakg_1}(H_1))+ \dim(\fraks)$; note that $\Lie([G, G])\cap \frakz(\frakg)= 0$ by Remarks~\ref{rems:fabulous}(iv).  Now $\dim(M)= \dim(\frakc_\frakg(H))$ since $C_G(H)$ is smooth, so $\dim(M_1)= \dim(\frakc_{\frakg_1}(H_1))$.  It follows that $C_{G_1}(H_1)$ is smooth.
 
 A similar argument shows that if $\frakh_1$ is a $k$-subalgebra of $\frakg_1$ then $C_{G_1}(\frakh_1)$ is smooth.
\end{proof}

\begin{rem}
\label{rem:levi_plong}
 Let $P$ be a $k$-parabolic subgroup of $G$ and let $L$ a Levi $k$-subgroup of $P$.  Let $\Gamma$ be a unipotent subgroup of $P(k)$ and suppose there is a parabolic $k$-subgroup $Q$ of $L$ such that $c_L(\Gamma)\subseteq R_u(Q)(k)$.  Now $QR_u(P)$ is a parabolic subgroup of $G$, and $\Gamma\subseteq R_u(Q)(k)R_u(P)(k)= (R_u(QR_u(P)))(k)$.  Hence $\Gamma$ is plongeable in $G$.  By a similar argument, if $\frakh$ is a $k$-subalgebra of $\Lie(P)$ consisting of nilpotent elements and $c_{\Lie(L)}(\frakh)\subseteq \Lie(R_u(Q))$ then $\frakh$ is plongeable in $G$.
\end{rem}

\begin{lem}
\label{lem:smooth_plong}
 Suppose the centraliser of every $k$-subgroup of $G$ and every $k$-subalgebra of $\frakg$ is smooth.  Then every unipotent subgroup of $G(k)$ is plongeable.
\end{lem}

\begin{proof}
 Let $\Gamma$ be a subgroup of $G(k)$ consisting of unipotent elements.  Without loss $\Gamma\neq 1$.  We use induction on $\dim(G)$.  The result is trivial if $\dim(G)= 0$.  Let $N$ be the algebraic subgroup of $G$ generated by $\Gamma$.  There is no harm in taking $\Gamma$ to be $N(k)$.  Let $Z$ be the last nontrivial term of the descending central series of $N$ (this makes sense as $N$ is unipotent).  Then $Z$ is nontrivial and $k$-defined (see \cite[Ch.\ 6f.]{milne}).  For any $u\in Z(k_s)$, $C_G(u)$ is smooth by hypothesis, so $u$ is $k_s$-plongeable by \cite[Prop.\ 2]{tits3}.  Hence by \cite[Cor.\ 3]{tits3} there is a $k_s$-defined parabolic subgroup $P$ of $G$ such that $Z\subseteq R_u(P)$.  We can take $P$ to be the subgroup $\mathscr{P}_{k_s}(Z(k_s))$ in the sense of \cite[Sec.\ 1]{tits3}; then $P$ is $k$-defined by \cite[Cor.\ 1]{tits3} and $N(k_s)\subseteq P(k_s)$ by \cite[Cor.\ 2]{tits3}.  Note that $P\neq G$ since $N\neq 1$.
 
 Let $L$ be a Levi $k$-subgroup of $P$.  By Lemma~\ref{lem:pretty_good_descent}(a) and our induction hypothesis, $c_L(\Gamma)$ is plongeable in $L$.  Hence $\Gamma$ is plongeable in $G$ by Remark~\ref{rem:levi_plong}.
\end{proof}

\begin{prop}
\label{prop:fab_crit}
 Suppose that centralisers of $k$-subgroups of $G$ and $k$-subalgebras of $\frakg$ are smooth.  Then for any $k$-subalgebra $\frakh$ of $\frakg$ consisting of nilpotent elements, there is a parabolic $k$-subgroup $P$ of $G$ such that $\frakh\subseteq \Lie(R_u(P))$.
\end{prop}

\begin{proof}
 We can assume without loss that $G$ is connected and $\frakh\neq 0$.  We use induction on $\dim(G)$.  The result holds trivially if $\dim(G)= 0$.  Let $\frakh$ be a $k$-subalgebra of $\frakg$ consisting of nilpotent elements.  Then $\frakh$ is a nilpotent Lie algebra by Engel's Theorem \cite[3.2 Thm.]{Hum72}, so $\frakz(\frakh)\neq 0$.  Pick $0\neq z\in \frakz(\frakh)$ and let $M= C_G(z)$.  By hypothesis $M$ is smooth, so $M$ is $k$-defined and $\frakh\subseteq \frakm:= \Lie(M)$.
 
 Let $Z= C_G(M)$.  Since $Z$ centralises $M$, $Z$ centralises $\frakm$; in particular, $Z$ centralises $z$.  Hence $Z\subseteq M$, and we deduce that $Z= Z(M)$.  It follows that $Z$ is commutative.  Since $Z$ is smooth, $z$ belongs to $\Lie(Z)$, so $\dim(Z)> 0$.  Fix a maximal $k$-torus $S$ of $Z$.  Since $M\subseteq C_G(S)$ and $\frakh\subseteq \frakm$, we can assume without loss by Remark~\ref{rem:plong_subgp} and Lemma~\ref{lem:pretty_good_descent}(a) that $G= C_G(S)$.  Let $\psi\colon G\to G/S$ be the canonical projection.  Then $\psi(Z)$ is unipotent, so by Lemma~\ref{lem:pretty_good_descent}(b) and Lemma~\ref{lem:smooth_plong} there is a parabolic $k$-subgroup $Q$ of $G/S$ such that $\psi(Z)\subseteq R_u(Q)$.  By \cite[Cor.\ 2]{tits3} we can assume that $\psi(M)\subseteq Q$.  So $M\subseteq P:= \psi^{-1}(Q)$, a $k$-parabolic subgroup of $G$ \cite[V.22.6 Thm.]{borel}.  Now $\psi(Z)$ is nontrivial since $d\psi(z)\neq 0$, so $Q$ is a proper parabolic subgroup of $G/S$; hence $P$ is proper.
 
 Let $L$ be a Levi subgroup of $P$.  Then $\dim(L)< \dim(P)$, so $c_{\Lie(L)}(\frakh)$ is plongeable in $L$ by our induction hypothesis.  It follows from Remark~\ref{rem:levi_plong} that $\frakh$ is plongeable in $G$, so we are done.
\end{proof}

We finish by showing that if $p$ is sufficiently large then $p$ is fabulous for $G$.

\begin{prop}
\label{prop:generic_fab}
 For $r\geq 0$, there is a prime $p_0= p_0(r)$ with the following property: for any prime $p\geq p_0$, any field $k$ of characteristic $p$ and any connected reductive $k$-group $G$ of rank at most $r$, $p$ is fabulous for $G$.
\end{prop}

\begin{proof}
 Since there are only finitely many possibilities for $G$ up to isomorphism (as a $\ovl{k}$-group), we can ensure that $p$ is very good for $G$ by taking $p$ sufficiently large.  Then $p$ is pretty good for $G$, and then --- as observed above --- condition (a) of Definition~\ref{def:fabulous} holds, which implies that condition (c) holds by Proposition~\ref{prop:fab_crit}.  Condition (b) holds for large enough $p$ by \cite[Thm.\ A]{HeSt}.  The result follows.
\end{proof}

\begin{rem}
 We do not have an analogous result to Proposition~\ref{prop:generic_fab} if we also require normalisers of subgroups to be smooth: see \cite[Lem.\ 11.11]{HeSt}.
\end{rem}

\section{Cocharacter-closed orbits and $G$-complete reducibility}
\label{sec:cochar}

We introduce the notion of complete reducibility for subalgebras of $\frakg$ and explain the link with geometric invariant theory (GIT).
At the end of the section, we extend to arbitrary $k$ the main results from \cite{mcninch}.
As in \cite{BMR:semisimplification}, our main tool from GIT is the notion of cocharacter-closure, introduced in \cite{GIT} and \cite{cochar}.

\begin{defn}
 Let $X$ be a $k$-defined affine $G$-variety and let $x\in X$ (we do not require $x$ to be a $k$-point). We say that the orbit $G(k)\cdot x$ is \emph{cocharacter-closed over $k$} if for all $\lambda\in Y_k(G)$ such that $x':= \lim_{a\to 0} \lambda(a)\cdot x$ exists, $x'$ belongs to $G(k)\cdot x$. 
 
 If $k= \ovl{k}$ then it follows from the Hilbert-Mumford Theorem \cite[Thm.~1.4]{kempf} that $G(k)\cdot x$ is cocharacter-closed over $k$ if and only if $G(k)\cdot x$ is closed. 
\end{defn}

The following is \cite[Thm.\ 1.3]{cochar}.

\begin{thm}[Rational Hilbert-Mumford Theorem]
\label{thm:RHMT}
 Let $G$, $X$, $x$ be as above. Then there is a unique $G(k)$-orbit ${\mathcal O}$ such that (a) ${\mathcal O}$ is cocharacter-closed over $k$, and (b) there exists $\lambda\in Y_k(G)$ such that $\lim_{a\to 0} \lambda(a)\cdot x$ belongs to ${\mathcal O}$.
\end{thm}

Next recall the notion of $G$-complete reducibility for subgroups of $G$.

\begin{defn}
	\label{def:gcr}
 Let $H$ be a subgroup of $G$. We say that $H$ is \emph{$G$-completely reducible over $k$} ($G$-cr over $k$) if for any parabolic $k$-subgroup $P$ of $G$ such that $P$ contains $H$, there is a Levi $k$-subgroup $L$ of $P$ such that $L$ contains $H$. 
 We say that $H$ is \emph{$G$-irreducible over $k$} ($G$-ir over $k$) if $H$ is not contained in any proper parabolic $k$-subgroup of $G$ at all.
 
 We say that $H$ is \emph{$G$-completely reducible} ($G$-cr) if $H$ is $G$-completely reducible over $\ovl{k}$.
\end{defn}

For more on $G$-complete reducibility for subgroups of $G$, see \cite{serre1.5}, \cite{serre2}, \cite{BMR}; our main focus in this paper is the analogue for Lie algebras. 
We first recall the definition of $G$-complete reducibility for Lie algebras and then also the link between this concept and GIT using generating tuples, due to Richardson.
Most of what follows was originally written down in \cite[\S 5]{sphericalcochar}.

\begin{defn}
	\label{def:relLieGcr}
	A subalgebra $\frakh$ of $\frakg$
	is \emph{$G$-completely reducible over $k$} ($G$-cr over $k$) if for any parabolic $k$-subgroup $P$ of $G$ such that $\frakh\subseteq \Lie(P)$,
	there is a Levi $k$-subgroup $L$ of $P$ such that $\frakh\subseteq \Lie(L)$ (see \cite[Def.\ 5.3]{sphericalcochar}). 
	We say that $\frakh$ is \emph{$G$-irreducible over $k$} ($G$-ir over $k$) if $\frakh$ is not contained in $\Lie(P)$ for any proper parabolic $k$-subgroup of $G$ at all.
	We say that $\frakh$ is \emph{$G$-indecomposable over $k$} ($G$-ind over $k$) if $\frakh$ is not contained in $\Lie(L)$ for any proper Levi $k$-subgroup $L$ of $G$.
	
	As in the subgroup case, we say that $\frakh$ 
	is \emph{$G$-completely reducible} (resp., \emph{$G$-irreducible}, \emph{$G$-indecomposable}) if it is $G$-completely reducible over $\ovl{k}$ (resp., $G$-irreducible over $\ovl{k}$, $G$-indecomposable over $\ovl{k}$).
\end{defn}

For $k = \ovl{k}$, this notion is due to McNinch, see \cite{mcninch} and also \cite[\S 5.3]{GIT}.

\begin{ex}
\label{ex:PGL_2_ir}
 Let $k$, $G$, $\frakh$ and $\frakm$ be as in Example~\ref{ex:PGL_2}.  Then $\frakh$ is $G$-ir but the ideal $\frakm$ of $\frakh$ is not $G$-cr.
\end{ex}

\begin{ex}
\label{ex:nilpt}
 If $p$ is fabulous for $G$ and $0\neq \frakh$ is a subalgebra of $\frakg$ consisting of nilpotent elements then $\frakh$ is not $G$-cr: for $\frakh\subseteq \Lie(B)$ for some Borel subgroup $B$ of $G$, but clearly $\frakh\not\subseteq \Lie(T)$ for any maximal torus $T$ of $B$.
\end{ex}

\begin{rem}
	\label{rem:linear}
	The notion of $G$-complete reducibility for subgroups and subalgebras generalizes the concept of semisimplicity from representation theory in the following sense. 
	If $H$ is a subgroup of $\GL(V)$ or $\SL(V)$, or $\frakh$ is a subalgebra of $\mathfrak{gl}(V)$ or $\mathfrak{sl}(V)$, then $H$ (resp.~$\frakh$) is $\GL(V)$-cr or $\SL(V)$-cr if and only if $V$ is a semisimple $H$-module (resp.~$\frakh$-module).
	See \cite{serre1.5}, \cite[(3.2.2)]{serre2} for the case of subgroups; the case of subalgebras is almost identical.  
\end{rem}

\begin{rem}
	\label{rem:absolute}
	Let $H$ be a subgroup of $G$ and $\frakh$ a subalgebra of $\frakg$. 
	If $k'/k$ is an algebraic field extension then we may also regard $G$ as a $k'$-group, and it therefore makes sense to ask whether 
	$H$ and $\frakh$ are $G$-cr over $k'$ as well as whether they are $G$-cr over $k$.
\end{rem}

\begin{rem}
	Note that Definitions~\ref{def:gcr} and \ref{def:relLieGcr} make sense even if $H$ and $\frakh$ are not $k$-defined. 
	We also note that since $\frakp_{g\cdot\lambda} = g\cdot \frakp_\lambda$ and $\frakl_{g\cdot\lambda} = g\cdot\frakl_\lambda$ for any $\lambda\in Y(G)$ and any $g\in G$ (see, e.g., \cite[\S 6]{BMR}), it follows that $\frakh$ is $G$-cr over $k$ if and only if every $\Ad(G(k))$-conjugate of $\frakh$ is. 
	More generally, one can show that if $\frakh$ is $G$-cr over $k$ (resp., $G$-ir over $k$, resp., $G$-ind over $k$) then so is $d\phi(\frakh)$, for any $k$-defined automorphism $\phi$ of $G$.
	Similar observations hold for subgroups.
\end{rem}

We now recall the link to GIT. For this, we need the following definition.

\begin{defn}\label{def:generating tuple}
	Let $\frakh$ be a Lie algebra. We call ${\mathbf x}=(x_1,\dots,x_m)\in \frakh^m$, for some $m\in \mathbb{N}$, a \emph{generating tuple for $\frakh$} if $x_1,\dots, x_m$ is a generating set for $\frakh$ as a Lie algebra.
\end{defn}

The next theorem shows the relevance of the previous definition to the study of complete reducibility. 
It is \cite[Thm.~5.4]{sphericalcochar}; see also \cite[Thm.~1(i)]{mcninch} for the case $k = \ovl{k}$.

\begin{thm}\label{thm:mcninch1}
	Let $\frakh$ be a subalgebra of $\frakg = \Lie(G)$. 
	Let $\tuple{x} \in \frakg^m$ be a generating tuple for $\frakh$, and let $G$ act on $\frakg^m$ by simultaneous conjugation.
	Then $\frakh$ is $G$-completely reducible over $k$ if and only if the $G(k)$-orbit of $\tuple{x}$ is cocharacter-closed in $\frakg^m$ over $k$.
\end{thm}

The next result is \cite[Thm.\ 5.5]{sphericalcochar}.
Note that if $k$ is perfect, then part (i) implies the equivalence of $G$-complete reducibility over $k$ and $G$-complete reducibility 
(over $\ovl{k}$), because the extension $\ovl{k}/k$ is separable for perfect $k$.

\begin{prop}
	\label{prop:ascent_descent}
Let $\frakh$ be a subalgebra of $\frakg$.
\begin{itemize}
\item[(i)] Suppose $\frakh$ is $k$-defined, and let $k'/k$ be a separable algebraic field extension. 
Then $\frakh$ is $G$-completely reducible over $k'$ if and only if $\frakh$ is $G$-completely reducible over $k$.
\item[(ii)] Let $S$ be a $k$-defined torus of $C_G(\frakh)$ and set $L = C_G(S)$. Then $\frakh$ is $G$-completely
reducible over $k$ if and only if $\frakh$ is $L$-completely reducible over $k$.
\end{itemize}
\end{prop}

The final ingredient we need for our first main result is the connection between complete reducibility of subgroups of $G$ 
and the notion of complete reducibility  in the spherical building of $G$.
Recall that the spherical building $\Delta_k = \Delta_k(G)$ of $G$ over $k$ can be identified as a simplicial complex with the poset of $k$-parabolic subgroups of $G$, with inclusion reversed \cite{tits1}; the parabolic subgroup $G$ corresponds to the empty simplex.
For each $k$-parabolic subgroup $P$ of $G$, we let $\sigma_P$ denote the corresponding simplex in $\Delta_k$.
We note that $k$-points of $G$ induce simplicial automorphisms of $\Delta_k$: for $g\in G(k)$ and a $k$-parabolic subgroup $P$ of $G$, we can define $g\cdot\sigma_P = \sigma_{gPg^{-1}}$.
Recall that the simplicial complex $\Delta_k$ has a \emph{geometric realisation}, which we denote by $\overline{\Delta_k}$ \cite[A.1.1]{abro}.
We have the following definitions.

\begin{defn}(\cite[2.1.4, 2.1.5, 2.2.1]{serre1.5})
\begin{itemize}
\item[(i)] Two simplices $\sigma, \tau$ in $\Delta_k$ are called \emph{opposite} if when we write $\sigma=\sigma_P$ and $\tau=\sigma_Q$ for $k$-parabolic subgroups $P$ and $Q$ of $G$, then $P$ and $Q$ are opposite in $G$; that is, $P\cap Q$ is a common Levi subgroup of $P$ and $Q$.
\item[(ii)] A subcomplex $\Sigma$ of $\Delta_k$ is called \emph{convex} if the corresponding subset $\overline{\Sigma}\subseteq \overline{\Delta_k}$ of the geometric realisation is convex.
\item[(iii)] A convex subcomplex $\Sigma$ of $\Delta_k$ is said to be \emph{$\Delta_k$-completely reducible} ($\Delta_k$-cr) if for every $\sigma\in \Sigma$, there exists $\tau\in \Sigma$ with $\tau$ opposite to $\sigma$.
\end{itemize}
\end{defn}

The following argument forms part of the proof of \cite[Lem.~4]{mcninch}; it provides the key link between $G$-complete reducibility and the building $\Delta_k$.
We give the idea of the proof for completeness.

\begin{prop}\label{prop:liebuildingcr}
Suppose $\frakh$ is a subalgebra of $\frakg$. Then  $\Delta_k^\frakh:=\{\sigma_P\mid \frak{h}\subseteq \Lie(P)\}$ is a convex subcomplex of $\Delta_k$, and it is $\Delta_k$-completely reducible if and only if $\frakh$ is $G$-completely reducible over $k$.
\end{prop}

\begin{proof}[Sketch Proof]
For the first assertion, the key observation is that for the intersection of any two connected subgroups $P$ and $Q$ of $G$ that are normalized by $T$, we have $\Lie(P\cap Q) = \Lie(P)\cap \Lie(Q)$ \cite[Cor. 13.21]{borel}.
Now the result follows from Serre's criterion \cite[Prop.~3.1]{serre2} for recognising a convex subcomplex of $\Delta_k$: a collection $\Sigma$ of simplices is a convex subcomplex if whenever $P,Q,R$ are $k$-parabolic subgroups of $G$ with $\sigma_P, \sigma_Q\in \Sigma$ and $P\cap Q\subseteq R$, then $\sigma_R\in \Sigma$.

To see that $\frakh$ is $G$-cr over $k$ if and only if $\Delta_k^\frakh$ is $\Delta_k$-cr,
the key again is the smoothness of the intersection of two $k$-parabolic subgroups.
This implies that for two $k$-parabolic subgroups $P$ and $Q$ of $G$, we can detect whether or not they are opposite (and hence correspond to opposite simplices of $\Delta_k$) on the level of their Lie algebras.
\end{proof}

\begin{rem}
 An analogous result holds for subgroups, with a very similar proof \cite{serre2}.
\end{rem}

Since each subalgebra $\frakh$ of $\frakg$ gives rise to a subcomplex $\Delta_k^\frakh$ of $\Delta_k$,
we can use the so-called Centre Conjecture of Tits, which in fact is a theorem for subcomplexes, see \cite{muhlherrtits}
($G$ of classical type or type $G_2$),
\cite{lrc}
($G$ of type $F_4$ or $E_6$) and \cite{rc}
($G$ of type $E_7$ or $E_8$).
We give a version which is suitable for our purposes.

\begin{thm}[Tits' Centre Theorem]\label{thm:TCC}
Let $\Delta_k$ be a thick spherical building, and $\Sigma$ a convex subcomplex of $\Delta_k$.
Then (at least) one of the following holds:
\begin{itemize}
\item[(i)] $\Sigma$ is $\Delta_k$-completely reducible; 
\item[(ii)] there exists a nonempty simplex $\sigma\in \Sigma$ which is fixed by all simplicial automorphisms of $\Delta_k$ that stabilize $\Sigma$.
\end{itemize}
\end{thm}

For a definition of \textit{thick} see \cite[4.2]{abro}. In our case of algebraic groups, $\Delta_k$ is always thick, see the comment after \cite[Thm.\ 3.2.6]{tits}. The typical use of this theorem (as we see in the next proof below) is in guaranteeing the existence of a simplex as in part (ii) when $\Sigma$ is \emph{not} $\Delta_k$-cr;
such a simplex is often referred to as a \emph{centre} of $\Sigma$.
We can now prove the main result of this section. 
Over $\ovl{k}$, this is \cite[Thm.~1(ii)]{mcninch}; see also \cite[Thm.~5.27, Ex.~5.29]{GIT}. 
However, the proofs given there, which use the technology of optimal destabilising subgroups, do not go through over arbitrary $k$.
Instead, we use the Centre Theorem.

\begin{thm}\label{thm:mcninch2}
Suppose $H$ is a $k$-defined subgroup of $G$.
If $H^0$ is $G$-completely reducible over $k$, then so is $\Lie(H)$. 
\end{thm}

\begin{proof}
Since $\Lie(H) = \Lie(H^0)$, and the hypothesis is that $H^0$ is $G$-cr, it is no loss to assume that $H$ is connected.
Let $\frakh = \Lie(H)$. 
Using Proposition \ref{prop:ascent_descent}(i) (and its analogue for subgroups \cite[Thm.~1.1]{BMR:sepext}), it is enough to prove the result over the separable closure $k_s$.
These reductions imply that we may assume that $H(k)$ is dense in $H$  \cite[11.2.7]{springer}.
Now suppose $L$ is a minimal Levi $k$-subgroup of $G$ containing $H$.
Since $L = L_\lambda = C_G(\Im(\lambda))$ for some $\lambda \in Y_k(C_G(H))$, 
Proposition \ref{prop:ascent_descent}(ii)
(and its analogue for subgroups \cite[Thm.~1.4]{sphericalcochar})
implies that $H$ and $\frakh$ are $G$-cr if and only if they are $L$-cr, so we may replace $G$ with $L$.
Then $H$ is $G$-cr over $k$ but is not contained in any proper Levi $k$-subgroup of $G$, so $H$ is $G$-ir over $k$.

We now proceed with the proof of the theorem by contradiction.
Suppose that $\frakh$ is not $G$-cr over $k$. 
The subgroup $H(k)$ acts on the building $\Delta_k$ by simplicial automorphisms, 
and the subcomplex $\Delta_k^\frakh$ is stabilized by the action of $H(k)$ since $H$ stabilizes its own Lie algebra.
We are assuming that $\frakh$ is not $G$-cr over $k$, so this subcomplex is not $\Delta_k$-cr, by Proposition \ref{prop:liebuildingcr},
and hence there is a nonempty $H(k)$-fixed simplex $\sigma$ in $\Delta_k^\frakh$ by Theorem \ref{thm:TCC}.
This simplex has the form $\sigma_P$ for some proper $k$-parabolic subgroup $P$ of $G$,
and the fact that $\sigma_P$ is $H(k)$-fixed translates into the fact that $P$ is normalized by $H(k)$.
Since parabolic subgroups are self-normalizing, this shows that $H(k)\subseteq P$, which in turn 
gives $H\subseteq P$ because $H(k)$ is dense in $H$.
This contradicts the conclusion of the first paragraph, that $H$ is $G$-ir over $k$, so we are done.
\end{proof}

\begin{rems}
\label{rems:subgp_to_subalg}
(i). We observe that the converse of Theorem \ref{thm:mcninch2} is false, already in the algebraically closed case.
See the counterexample due to the third author in \cite[\S 1]{mcninch}; see also Example \ref{ex:ssoverkvskbar} below.

(ii). If $k$ is perfect, then Theorem \ref{thm:mcninch2} also holds with the hypothesis that $H$ (instead of $H^0$) is $G$-cr. 
For, as has already been observed, if $k$ is perfect then
$H$ is $G$-cr over $k$ if and only if $H$ is $G$-cr, and the same for $\Lie(H)$. 
Now over $\ovl{k}$, if $H$ is $G$-cr then $H^0$ is $G$-cr too \cite[Thm.~3.10]{BMR}, and the hypotheses of the theorem hold.
\end{rems}

\section{Maps induced by inclusion}

In this section we assume $k$ is algebraically closed.  We need some material on quotients in GIT.  Recall that if $G$ acts on an affine variety $X$ then we can form the quotient variety $X/\!\!/G$.  The coordinate algebra $k[X/\!\!/G]$ is by definition the subring of invariants $k[X]^G$ of the coordinate ring $k[G]$, and the inclusion of $k[X]^G$ in $k[G]$ induces a morphism $\pi_X$ (or $\pi_{X,G}$) from $X$ to $X/\!\!/G$.  If $x\in X$, there is a unique closed $G$-orbit $C(x)$ contained in the closure of the orbit $G\cdot x$.  By the Hilbert-Mumford Theorem, there exists $\lambda\in Y(G)$ such that $\lim_{a\to 0} \lambda(a)\cdot x$ belongs to $C(x)$.  Given $x_1, x_2\in X$, we have $\pi_X(x_1)= \pi_X(x_2)$ if and only if $C(x_1)= C(x_2)$ if and only if $f(x_1)= f(x_2)$ for all $f\in k[X]^G$.  In particular, if $G\cdot x_1$ and $G\cdot x_2$ are closed then $C(x_1)= G\cdot x_1$ and $C(x_2)= G\cdot x_2$, so $\pi_{X}(x_1)= \pi_X(x_2)$ if and only if $x_1$ and $x_2$ lie in the same $G$-orbit.  Hence points of $X/\!\!/G$ correspond to closed $G$-orbits in $X$.  For background on quotient varieties and GIT, see \cite[Ch.\ 3]{newstead} or \cite{BaRi}.  Here is a particular instance of this set-up.  The group $G$ acts on $G^m$ by simultaneous conjugation, so we can form the quotient $G^m/\!\!/G$.  Likewise we can form the quotient $\frakg^m/\!\!/G$.  There is a direct connection with $G$-complete reducibility arising from Theorem~\ref{thm:mcninch1} as follows: points in $\frakg^m/\!\!/G$ correspond to closed $G$-orbits $G\cdot (x_1,\ldots, x_m)$, and $(x_1,\ldots, x_m)\in \frakg^m$ yields a closed orbit if and only if the subalgebra $\frakh$ generated by the $x_i$ is $G$-cr.  Analogous statements hold for $G^m/\!\!/G$.

Now let $H$ be a reductive subgroup of $G$ and let $\frakh= \Lie(H)$.  The inclusion $\iota$ of $H^m$ in $G^m$ induces a map $\Psi:H^m/\!\!/H\to G^m/\!\!/G$.  The third author proved that $\Psi$ is a finite morphism of varieties \cite[Thm.\ 1.1]{martin1}; this has various consequences for the theory of $G$-completely reducible subgroups (see, e.g., \cite[Cor.\ 3.8]{BMR}).  It follows that the image of $\Psi$ is closed.  In particular, $\pi_{G^m}(\iota(H^m))$ is a closed subset of $G^m/\!\!/G$.
Now consider the analogous situation for Lie algebras. The differential $d\iota$ of $\iota$ maps $\frakh^m$ to $\frakg^m$.  Now $d\iota$ gives rise to a map
$$ \psi\colon \frakh^m/\!\!/H\to \frakg^m/\!\!/G $$
mapping $\pi_{\frakh^m, H}(x_1,\ldots, x_m)$ to $\pi_{\frakg^m, G}(d\iota(x_1,\ldots, x_m))$, but we see in Example~\ref{ex:PGL_2.2} below that this map need not be finite.

First we need a preliminary result. 

\begin{prop}
\label{prop:nullcone}
 Let $(x_1,\ldots, x_m)\in \frakg^m$ and let $\frakm$ be the subalgebra spanned by the $x_i$.  Then the following are equivalent:
 \begin{itemize}
  \item[(a)] $\pi_{\frakg^m, G}(x_1,\ldots, x_m)= \pi_{\frakg^m, G}(0,\ldots, 0)$.
  \item[(b)] For every nonconstant homogeneous $f\in k[\frakg^m]^G$, we have $f(x_1,\ldots, x_m)= 0$.
  \item[(c)] There exists $\lambda\in Y(G)$ such that $\lim_{a\to 0} \lambda(a)\cdot (x_1,\ldots, x_m)= (0,\ldots, 0)$.
  \item[(d)] There exists a maximal unipotent subgroup $U$ of $G$ such that $\frakm\subseteq \Lie(U)$.
 \end{itemize}
\end{prop}

\begin{proof}
 The equivalence of (a), (b) and (c) follows from the results given at the start of the section.  The equivalence of (c) with (d) follows from Lemma~\ref{lem:liealgebrasofRpars}(iii), since a  maximal unipotent subgroup of $G$ is the unipotent radical of a Borel subgroup.
\end{proof}

\begin{ex}
\label{ex:PGL_2.2}
 Let $p= 2$ and let $G= \SL_3(k)$.  Let $H= \PGL_2(k)$.
 The adjoint representation of $\SL_2(k)$ on its Lie algebra gives rise to an embedding $i$ of $H$ in $G$; with a suitable choice of basis, this takes the form
 $i\left(\ovl{\twobytwo{a}{b}{c}{d}}\right)= \threebythree{1}{ac}{bd}{0}{a^2}{b^2}{0}{c^2}{d^2}$.
 
 For $a\neq 0$, define $x_1= \ovl{\twobytwo{0}{1}{0}{0}}$ and $x_2(a)= \ovl{\twobytwo{0}{0}{a}{0}}$.  Then each pair $(x_1, x_2(a))$ spans the subalgebra $\frakh$ from Example~\ref{ex:PGL_2}, which is $H$-ir by Example~\ref{ex:PGL_2_ir}.  It follows from Theorem~\ref{thm:mcninch1} that $H\cdot (x_1, x_2(a))$ is closed for each $a$.  It is easily seen that the $(x_1, x_2(a))$ are pairwise non-conjugate under the $H$-action.  Hence the points $\pi_{\frakh^2, H}(x_1, x_2(a))$ are pairwise distinct in $\frakh^2/\!\!/H$.  On the other hand, $di(\frakh)\subseteq \Lie(B)$, where $B$ is the Borel subgroup of upper triangular matrices in $G$, so $\pi_{\frakg^2, G}(x_1, x_2(a))= \pi_{\frakg^2, G}(0, 0)$ for all $a$, by Proposition \ref{prop:nullcone}.  This shows that the fibre $\psi^{-1}(\pi_{\frakg^2, G}(0,0))$ is infinite, so $\psi$ cannot be finite.
\end{ex}

If we put a restriction on $p$ then the situation improves.  Recall that if $V$ is a rational $G$-module then $v\in V$ is said to be \emph{unstable} if 0 belongs to the closure of $G\cdot v$, and \emph{semistable} otherwise (cf.\ \cite[Sec.\ 1.4]{rich}).

\begin{lem}
\label{lem:gdd_nakayama}
 Let $R$ and $S$ be commutative rings graded by $\NN_0$ with $R_0= S_0= k$.  Set $R_+= \sum_{i> 0} R_i$ and $S_+= \sum_{i> 0} S_i$.  Let $\phi\colon R\to S$ be a graded $k$-algebra homomorphism; we regard $S$ as an $R$-module via $\phi$.  Let $I$ be $S\phi(R_+)$ (regarded as an ideal of $S$).  Suppose $S$ is finitely generated as a $k$-algebra and $S_+$ is the radical of $I$.  Then $S$ is a finitely generated $R$-module.
\end{lem}

\begin{proof}
 Choose homogeneous elements $f_1,\ldots, f_r\in S_+$ such that the $f_i$ generate $S$ as a $k$-algebra.  There exist $n_1,\ldots, n_r\in \NN$ such that $f_i^{n_i}\in I$ for each $i$.  The elements of the form $f_1^{a_1}\cdots f_r^{a_r}$, where the $a_i $ are non-negative integers, span $S$ over $k$.  If $a_i\geq n_i$ for some $i$ then $f_1^{a_1}\cdots f_r^{a_r}\in I$.  Hence $S/I$ is spanned over $k$ by the images of the elements of the form $f_1^{a_1}\cdots f_r^{a_r}$ with $0\leq a_i< n_i$.  This shows that $S/I$ is finitely generated as an $R$-module. 
 This shows that $S/I$ is finitely generated as an $R$-module.  It follows from a corollary \cite[Cor.\ 3]{markwig} to the graded Nakayama lemma that $S$ is a finite $R$-module; note that Lemma 1 of \emph{op.\ cit.}\ still holds even when the module $M$ is not finitely generated (see \cite[Ex.~4.6]{eisenbud}). This completes the proof.
\end{proof}

\begin{prop}
\label{prop:finite}
Suppose $p$ is fabulous for $G$.  Then $\psi$ is a finite map.
\end{prop}

\begin{proof}
 Suppose $(x_1,\ldots, x_m)\in \frakh^m$ is unstable for the $G$-action.  Then the subalgebra $\frakm$ generated by the $x_i$ consists of nilpotent elements, by Proposition~\ref{prop:nullcone}.  The hypothesis on $p$ implies that $\frakm\subseteq \Lie(U)$ for some maximal unipotent subgroup $U$ of $H$.  Hence $(x_1,\ldots, x_m)$ is unstable for the $H$-action by Proposition~\ref{prop:nullcone}.  It follows from Proposition~\ref{prop:nullcone} again that
 \begin{equation}
 \label{eqn:zero_fibre}
  \psi^{-1}(\pi_G(0,\ldots, 0))= \{\pi_H(0,\ldots, 0)\}.
 \end{equation}
 Now $k[\frakg^m]$ is a polynomial algebra, so it has a natural grading by $\NN_0$.  Since $G$ acts linearly on $\frakg^m$, the ring of invariants $R:= k[\frakg^m]^G= k[\frakg^m/\!\!/G]$ is $\NN_0$-graded.  Likewise, $S:= k[\frakh^m]^H= k[\frakh^m/\!\!/H]$ is $\NN_0$-graded, and the comorphism $\psi^*\colon R\to S$ is graded because the inclusion of $\frakh^m$ in $\frakg^m$ is linear.  It follows from Eqn.\ (\ref{eqn:zero_fibre}) that $S_+$ is the radical of $S\psi^*(R_+)$, so $S$ is a finite $R$-module by Lemma~\ref{lem:gdd_nakayama}.  Hence $\psi$ is finite.
\end{proof}

\section{$k$-semisimplification}
\label{sec:ss}

First, we recall the notion of $k$-semisimplification for subgroups of $G$ and the main theorem from \cite{BMR:semisimplification}.

\begin{defn}
\label{defn:ss}
 Let $H$ be a subgroup of $G$. We say that a subgroup $H'$ of $G$ is a \emph{$k$-semisimplification of $H$ (for $G$)} if there exist a parabolic $k$-subgroup $P$ of $G$ and a Levi $k$-subgroup $L$ of $P$ such that $H\subseteq P$, $H'= c_L(H)$ and $H'$ is $G$-completely reducible (or equivalently $L$-completely reducible, by \cite[Prop.~3.6]{BMR:semisimplification}) over $k$, where $c_L:P\to L$ is the canonical projection. We say \emph{the pair $(P,L)$ yields $H'$}.
\end{defn}

\begin{rems}
\label{rem:ss}
 (i). Let $H$ be a subgroup of $G$. If $H$ is $G$-cr over $k$ then clearly $H$ is a $k$-semisimplification of itself, yielded by the pair $(G,G)$. 
 
 (ii). Given any subgroup $H$ of $G$,  \cite[Rem.~4.3]{BMR:semisimplification} guarantees the existence of a $k$-semi-simplification of $H$.
\end{rems}

Here is the main result \cite[Thm.~ 4.5]{BMR:semisimplification} from \cite{BMR:semisimplification}, which was proved in the special case $k= \ovl{k}$ in \cite[Prop.~5.14(i)]{GIT}, cf.\ \cite[Prop.~3.3(b)]{serre2}. 
The uniqueness asserted in Theorem~\ref{thm:main} is akin to the theorem of Jordan--H\"older.

\begin{thm}
\label{thm:main}
 Let $H$ be a subgroup of $G$. Then any two $k$-semisimplifications of $H$ are $G(k)$-conjugate.
\end{thm}

We now come to the analogue of Definition \ref{defn:ss} for subalgebras of $\frakg$. 

\begin{defn}
	\label{defn:liess}
	Let $\frakh$ be a Lie subalgebra of $\frakg$. We say that a Lie subalgebra $\frakh'$ of $\frakg$ is a \emph{$k$-semisimplification of $\frakh$ (for $G$)} if there exist a parabolic $k$-subgroup $P$ of $G$ and a Levi $k$-subgroup $L$ of $P$ such that $\frakh \subseteq \Lie(P)$, $\frakh'= c_{\Lie(L)}(\frakh)$ and $\frakh'$ is $G$-completely reducible (or equivalently, by Proposition~\ref{prop:ascent_descent}(ii), $L$-completely reducible) over $k$. We say \emph{the pair $(P,L)$ yields $\frakh'$}.
\end{defn}

\begin{rems}
	\label{rem:liess}
		(i). Let $\frakh$ be a subalgebra of $\frakg$. If $\frakh$ is already $G$-cr over $k$ then clearly $\frakh$ is a $k$-semisimplification of itself, yielded by the pair $(G,G)$. 
		
		(ii). Suppose $(P,L)$ yields a $k$-semisimplification $\frakh'$ of $\frakh$. Let $L_1$ be another Levi $k$-subgroup of $P$. Then $L_1= uLu^{-1}$ for some $u\in R_u(P)(k)$ by Lemma~\ref{lem:parprops}(iii), so consequently $c_{\Lie(L_1)}(\frakh)= u\cdot c_{\Lie(L)}(\frakh)$. Hence $(P,L_1)$ also yields a $k$-semisimplification of $\frakh$. 
		Because of this, when the choice of $L$ doesn't matter we simply say that \emph{$P$ yields a $k$-semisimplification of $\frakh$}.
		
		(iii). It is straightforward to check that if $\phi$ is an automorphism of $G$ (as a $k$-group), $\frakh$ is a subalgebra of $\frakg$ and $(P,L)$ yields a $k$-semisimplification $\frakh'$ of $\frakh$ then $d\phi(\frakh')$ is a $k$-semisimplification of $d\phi(\frakh)$, yielded by $(\phi(P), \phi(L))$.
\end{rems}

The following is immediate from Lemma \ref{lem:P,L is Pl, Ll}.

\begin{lem}
	\label{lem:Liek-cochar}
	Suppose that $\frakh'$ is a $k$-semisimplification of $\frakh$.
	Then there is $\lambda\in Y_k(G)$ such that $\frakh'$ is yielded by the pair $(P_\lambda,L_\lambda)$.
\end{lem}

As in the group case (Remark \ref{rem:ss}(ii)) we always have the existence of a $k$-semisimplification of an arbitrary subalgebra of $\frakg$, due to the rational Hilbert-Mumford Theorem \ref{thm:RHMT}, as the following remark shows.

\begin{rem}
	\label{rem:lieexistence}
	Suppose $\frakh$ is a subalgebra of $\frakg$.
	Let
	${\mathbf h}= (h_1,\ldots, h_m)\in \frakh^m$ be a generating tuple for $\frakh$. 
	Then 
	 $c_\lambda({\mathbf h})= (c_\lambda(h_1),\ldots, c_\lambda(h_m))$ is a generating tuple for $c_\lambda(\frakh)$, for any $\lambda \in Y_k(G)$, and hence $c_\lambda(\frakh)$ is a $k$-semisimplification of $\frakh$ if and only if $G(k)\cdot c_\lambda({\mathbf h})$ is cocharacter-closed over $k$, by Theorem~\ref{thm:mcninch1}. It follows from Theorem~\ref{thm:RHMT} that $\frakh$ admits at least one $k$-semisimplification: for we can choose $\lambda\in Y_k(G)$ such that $G(k)\cdot c_\lambda({\mathbf h})$ is cocharacter-closed over $k$, so $c_\lambda(\frakh)$ is a $k$-semisimplification of $\frakh$, yielded by $(P_\lambda, L_\lambda)$.
\end{rem}

Here is the analogue of the main result Theorem 4.5 from \cite{BMR:semisimplification} in the Lie algebra setting, which 
can be viewed as a kind of Jordan--H\"older theorem.
Since the adjoint action is $k$-linear, the proof is easier to the one in \cite{BMR:semisimplification}, where a descending chain argument is needed.

\begin{thm}
	\label{thm:liemain}
	Let $\frakh$ be a subalgebra of $\frakg$.
	Then any two $k$-semisimplifications of $\frakh$ are $\Ad(G(k))$-conjugate.
\end{thm}

\begin{proof}
 Let $\frakh_1, \frakh_2$ be $k$-semisimplifications of $\frakh$. By Lemma~\ref{lem:Liek-cochar}, there exist $\lambda_1, \lambda_2\in Y_k(G)$ such that $(P_{\lambda_1}, L_{\lambda_1})$ realizes $\frakh_1$ and $(P_{\lambda_2}, L_{\lambda_2})$ realizes $\frakh_2$. Let ${\mathbf x}\in \frakh^m$ be a generating tuple for $\frakh$. Then $c_{\lambda_i}({\mathbf x})$ is a generating tuple for $\frakh_i$ for $i= 1,2$, and each orbit $G(k)\cdot c_{\lambda_i}({\mathbf x})$ is cocharacter-closed over $k$.
 It follows from the uniqueness result in the rational Hilbert-Mumford Theorem \ref{thm:RHMT} that the two orbits $G(k)\cdot c_{\lambda_1}({\mathbf x})$ and $G(k)\cdot c_{\lambda_2}({\mathbf x})$ have to be equal. Thus there exists $g\in G(k)$ such that $g\cdot c_{\lambda_1}({\mathbf x})= c_{\lambda_2}({\mathbf x})$. This means that the spanning set $c_{\lambda_1}({\mathbf x})$ of $\frakh_1$ is $\Ad(G(k))$-conjugate to the one of $\frakh_2$. Since the adjoint action is $k$-linear, $\frakh_1$ and $\frakh_2$ are also $\Ad(G(k))$-conjugate.
\end{proof}

Next we study the connection between the notions of $k$-semisimplifications for subgroups and subalgebras. It turns out that they are compatible in a natural fashion.

\begin{thm}
	\label{thm:ssLievsGp}
	Let $H$ be a subgroup of $G$ and let $H'$ be a $k$-semisimplification of $H$.
	Then 
	$\Lie(H')$ is a $k$-semisimplification of $\Lie (H)$.
\end{thm}

\begin{proof}
	By Lemma \ref{lem:P,L is Pl, Ll}, there is a $\lambda \in Y_k(G)$ such that $(P_\lambda,L_\lambda)$ yields $H' = c_\lambda(H)$. 
		It follows from Theorem \ref{thm:mcninch2} that
	$\Lie(H')$ is $G$-cr over $k$. Since the differential of conjugation is the adjoint action, we have $dc_{L_\lambda} = c_{\frakl_\lambda}$. Using Lemma \ref{lem:liealgebrasofRpars}, it follows that $\Lie(H') = c_{\frakl_\lambda}(\Lie(H))$, and so the pair $(P_\lambda,L_\lambda)$ also yields $\Lie(H')$.
\end{proof}

Next we revisit the example of Remark~\ref{rems:subgp_to_subalg}(i)
in the context of Theorems \ref{thm:liemain} and \ref{thm:ssLievsGp}.

\begin{ex}\label{ex:Ben}
Suppose $\char(k) = p >0$. Let
$H$ be a non-trivial connected 
semisimple group and let $\varrho_i : H \to \SL(V_i)$ be representations of $H$ for $i = 1,2$ with $\varrho_1$ semisimple and $\varrho_2$ not semisimple. 
Let $\varrho : H \to G:= \SL(V_1 \oplus V_2)$
be the representation given by 
$h \mapsto \varrho_1 (h) \oplus \varrho_2 (F (h))$, where
$F : H \to H$ is the Frobenius endomorphism of $H$. 
Let $J$ be the image of $H$ under $\varrho$ in $G$.
Since $V_1 \oplus V_2$ is not semisimple as a $J$-module, $J$ is not $G$-cr (cf.~Remark \ref{rem:linear}).
However, $\Lie(J) = \Im\, d\varrho_1 \oplus 0 \subseteq \frakg$ \emph{is} $G$-cr, showing that
the converse of Theorem \ref{thm:mcninch2} fails.

Now let $J'$ be a $k$-semisimplification of $J$. 
It follows from Theorem \ref{thm:ssLievsGp}
that $\Lie(J')$ is a $k$-semisimplification of $\Lie(J)$ and thus, by Remark \ref{rem:liess}(a) and Theorem \ref{thm:liemain}, that 
$\Lie(J')$ and $\Lie(J)$ are $\Ad(G(k))$-conjugate.
\end{ex}

The following is the analogue of \cite[Def.~4.6]{BMR:semisimplification}.

\begin{defn}
	Let $\frakh$ be a subalgebra of $\frakg$.
 We define ${\mathcal D}_k(\frakh)$ to be the set of $\Ad(G(k))$-conjugates of any $k$-semisimplification of $\frakh$ in $\frakg$. This is well-defined by Theorem~\ref{thm:liemain}.
\end{defn}

In the following two examples we show that not every element of ${\mathcal D}_k(\frakh)$ need be a $k$-semisimplification of $\frakh$ and that there is no direct relation between the notions of $k$-semisimplification and $\ovl{k}$-semisimplification of a subalgebra of $\frakg$.

\begin{ex}
	\label{ex:liecomments}
 Let $\frakh$ be a subalgebra of $\frakg$. As noted in Remark \ref{rem:liess}(a), if $\frakh$ is $G$-cr over $k$ then $\frakh$ is a $k$-semisimplification of itself, yielded by the pair $(G,G)$. 
 If $\frakh$ is $G$-ir, then $\frakh$ is not contained in $\Lie(P)$ for any proper parabolic subgroup $P$ of $G$, so $\frakh$ is the {\bf only} $k$-semisimplification of itself.  
\end{ex}

\begin{ex}
\label{ex:ssoverkvskbar}
There are many examples in the literature of the following: a reductive group $G$ over an imperfect field $k$, and a subgroup
$H$ of $G$ such that $H$ is $G$-cr over $k$ but not $G$-cr, or $H$ is $G$-cr but not $G$-cr over $k$; see \cite[Thm.~1.3]{BannuscherLitterickUchiyama}, for example.
Not all of these instances give rise to similar ones on the level of Lie algebras, even when the subgroup $H$ is connected, 
because of problems like that in Example \ref{ex:Ben}.
However, one of the most basic families of examples does work, as we now describe.

Let $k$ be an imperfect field of characteristic $p$, and let $a\in k\setminus k^p$.
Let $t$ be a $p^{\rm th}$ root of $a$ in $\ovl{k}$; then the extension $k'=k(t)$ is purely inseparable over $k$ of degree $p$.
Given a $k'$-group $H'$, we may form the \emph{Weil restriction} $\mathrm{R}_{k'/k}(H')$, which is 
a $k$-group, see \cite[A.5]{CGP} for example.
(It is easiest to describe Weil restriction functorially: if we view $H'$ as a functor from $k'$-algebras to groups,
then $\mathrm{R}_{k'/k}(H')$ is the corresponding functor from $k$-algebras to groups with $\mathrm{R}_{k'/k}(H')(A):= H'(A\otimes_k k')$ for each $k$-algebra $A$.)
Now let $H' = \Gm$ be the multiplicative group over $k'$.
Then $H := \mathrm{R}_{k'/k}(\Gm)$ is a $p$-dimensional abelian $k$-group with a $(p-1)$-dimensional
unipotent radical, but no connected normal unipotent $k$-subgroup (it is a basic example of a so-called \emph{pseudo-reductive group}, see \cite[Ex.~1.1.3]{CGP}).
The natural action of $H'$ on $k'$ by multiplication Weil restricts to give a $p$-dimensional representation of $H$ over $k$
which is irreducible: in terms of coordinates, this arises by writing down a $k$-basis for $k'$ and
interpreting the action through that basis, see \cite[\S 5.2]{BS}.
However, after base changing to $\ovl{k}$, this module becomes indecomposable and not irreducible \cite[Rem.~4.7(i)]{BS}.
This means that $H$ is $\GL_p$-cr over $k$ but not $\GL_p$-cr.  This example is due to McNinch (see \cite[Ex.\ 5.11]{BMR}).

If we turn our attention to the Lie algebras, we may identify the Lie algebra $\Lie(H')$ with the additive group over $k'$,
and we may view the multiplicative group $H'$ as an open subset.
Doing this is compatible with the action of $H'$ and $\Lie(H')$ on $k'$ by multiplication.
Therefore, after Weil restricting, we may identify the Lie algebra of $H$ with a $p$-dimensional abelian Lie algebra over $k$ 
containing $H$ as an open subset, with the same compatibility between the corresponding representations (see \cite[A.7.6]{CGP} for more details on
the Lie algebra of a Weil restriction).
Therefore, in this case, we also have that $\Lie(H)$ is $\GL_p$-cr over $k$ but not $\GL_p$-cr.  Note that $C_{\GL_p}(\Lie(H))$ is smooth.
\end{ex}

The following theorem shows that under a mild restriction on $k$, the process of $k$-semi-simplification behaves well under passing to ideals.

\begin{thm}
	\label{thm:lienormal}
	Suppose $p$ is fabulous for $G$.  Let $\frakh$ be a subalgebra of $\frakg$ and let $\frakm$ be an ideal of $\frakh$. Then:
	\begin{itemize}
		\item[(a)] If $\frakh$ is $G$-completely reducible over $k$ then so is $\frakm$.
		\item[(b)] Every parabolic subgroup $P$ of $G$ which yields a $k$-semisimplification of $\frakh$ also yields one for $\frakm$. In particular, there exist $k$-semisimplifications $\frakh'$ of $\frakh$ and $\frakm'$ of $\frakm$ such that $\frakm'$ is an ideal in $\frakh'$.
	\end{itemize}
\end{thm}

\begin{proof}
	For (a), by the same reductions as at the start of the proof of Theorem \ref{thm:mcninch2}, we may reduce to the case that $k=k_s$ and $\frakh$ is $G$-ir over $k$.
	We again proceed from this point using a contradiction argument invoking Tits' Centre Theorem \ref{thm:TCC}.
	So suppose $\frakm$ is not $G$-cr over $k$. Then the subcomplex $\Delta_k^\frakm$ of the building $\Delta_k$ is not $\Delta_k$-cr, by Proposition \ref{prop:liebuildingcr}, and hence there is a proper $k$-parabolic subgroup $P$ of $G$ such that $\sigma_P\in \Delta_k^\frakm$ is
	fixed by all simplicial automorphisms of $\Delta_k$ stabilizing $\Delta_k^\frakm$, by Theorem \ref{thm:TCC}.
	Since $N_G(\frakm)(k)$ clearly stabilizes $\Delta_k^\frakm$, we can conclude that $N_G(\frakm)(k)\subseteq P$.
	Now, because $p$ is fabulous, $N_G(\frakm)$ is smooth, and because $k=k_s$ this means that the $k$-points of $N_G(\frakm)$ are dense 
	by \cite[AG.13.3 Cor.]{borel}. 
	Therefore, we may conclude that $N_G(\frakm)\subseteq P$.
	Finally, this gives 
	$$
	\frakm \subseteq \frakh \subseteq \frakn_\frakg(\frakm)= \Lie(N_G(\frakm))
	\subseteq \Lie(P),
	$$ 
	where the
	equality in the middle follows from smoothness of $N_G(\frakm)$ again.
	This gives the required contradiction, as we had reduced to the case that $\frakh$ is $G$-ir over $k$.
		
	For (b), pick any $\lambda\in Y_k(G)$ such that $(P_\lambda, L_\lambda)$ yields a $k$-semisimplification $\frakh':= c_\lambda(\frakh)$ of $\frakh$. Then $c_\lambda(\frakh)$ is $G$-cr over $k$ and, as $c_\lambda$ is a Lie algebra homomorphism, $c_\lambda(\frakm)$ is an ideal in $ c_\lambda(\frakh)$. Now $c_\lambda(\frakh)$ and $c_\lambda(\frakm)$ satisfy the hypotheses of the theorem,
	so $c_\lambda(\frakm)$ is $G$-cr over $k$ by (a). Hence $(P_\lambda, L_\lambda)$ yields a semisimplification $\frakm':= c_\lambda(\frakm)$ of $\frakm$ as well, and $\frakm'$ is an ideal in $\frakh'$.
\end{proof}

\begin{rem}
Both parts of Theorem~\ref{thm:lienormal} are false without the assumption on $p$: see Example~\ref{ex:PGL_2_ir}.
\end{rem}

\section{$G$-toral and solvable subalgebras}
\label{sec:toral_solvable}

We assume in this section that $k$ is algebraically closed.  We study $G$-complete reducibility properties of solvable and $G$-toral subalgebras $\frakh$ of $\frakg$ (see Definition~\ref{defn:toral} for the latter).

\begin{defn}
 We call a Lie subalgebra $\frakh$ of $\frakg$ \emph{Jordan-closed} if  for every $x\in \frakh$, the semisimple part $x_s$ and nilpotent part $x_n$ of $x$ both belong to $\frakh$.\footnote{For Lie subalgebras of $\mathfrak {gl}_n$ over a field of characteristic
 	0, the notion of ``Jordan-closed'' already appears in \cite[Ch.~VII, \S 5]{Bou05} under the term ``scindable''.}  We define the \emph{Jordan closure} $\frakh^J$ of $\frakh$ to be the smallest Jordan-closed Lie subalgebra of $\frakg$ that contains $\frakh$.
\end{defn}

\begin{rems}
\label{rems:closure_invt}
 (a). It is clear that $\frakh^J$ is well-defined.  Here is an explicit construction.  We define an increasing chain of subalgebras $\frakh_i$ of $\frakg$ as follows.  Set $\frakh_0= \frakh$.  Given $\frakh_i$, let $\frakh_{i+1}$ be the subalgebra generated by the elements of the form $x_s$ and $x_n$ for $x\in \frakh_i$.  For dimension reasons, the chain becomes stationary and we have $\frakh_n= \frakh^J$ for $n$ sufficiently large.
 
 (b). If $\frakh$ is algebraic then $\frakh$ is Jordan-closed.  In particular, if $P$ is a parabolic subgroup of $G$ and $L$ is a Levi subgroup of $P$ then $\Lie(P)$ and $\Lie(L)$ are Jordan-closed.  It follows easily that for any subalgebra $\frakm$ of $\frakg$, $\frakm$ is $G$-cr (resp., $G$-ir, resp., $G$-ind) if and only if $\frakm^J$ is.  Also, if $\char(k)= 0$, $\frakh$ is a subalgebra of $\frakg$ and $\frakh$ is semisimple then $\frakh$ is algebraic (see \cite[Lem.\ 3.2]{rich}), so $\frakh$ is Jordan-closed.
\end{rems}

\begin{rem}
\label{rem:Jordan_funct}
 Let $\frakh$ be a subalgebra of $\frakg$ and let $f\colon G\to M$ be a homomorphism of connected reductive groups.  We have $df(x)_s= df(x_s)$ and $df(x)_n= df(x_n)$ for any $x\in \frakh$, so $df(\frakh)$ is Jordan-closed if $\frakh$ is. The converse obviously holds if $f$ is an embedding, for then $df$ is injective.
\end{rem}

\begin{defn}
\label{defn:toral}
 A subalgebra $\frakh$ of $\frakg$ is \emph{$G$-toral} if $\frakh\subseteq \Lie(S)$ for some torus $S$ of $G$.
\end{defn}

\begin{rem}
\label{rem:toral}
 Recall that a Lie algebra $\frakh$ is said to be \emph{toral} if every element of $\frakh$ is $\ad_\frakh$-semisimple; in this case, $\frakh$ is abelian \cite[Lem.\ 8.1]{Hum72}.  Clearly, if $\frakh$ is $G$-toral then $\frakh$ is toral, but the converse is false: e.g., take $\frakh$ to be a nonzero abelian subalgebra consisting of nilpotent elements.
\end{rem}

The following is the counterpart of 
\cite[Lem.~11.24]{Jantzen} for subalgebras  of $\frakg$.

\begin{lem}
\label{lem:toral_Gcr}
 Let $\frakh$ be a $G$-toral subalgebra of $\frakg$.  Then $\frakh$ is $G$-completely reducible.
\end{lem}

\begin{proof}
 Choose a torus $S$ such that $\frakh\subseteq \Lie(S)$.  By Proposition~\ref{prop:ascent_descent}(ii) it suffices to prove that $\frakh$ is $C_G(S)$-cr.  But this is clear because $S$ is central in $C_G(S)$, so $\Lie(S)$ is contained in $\Lie(P)$ and $\Lie(L)$ for every parabolic subgroup $P$ and every Levi subgroup $L$ of $C_G(S)$.
\end{proof}

\begin{rem}
\label{rem:toral_normaliser}
Let $\frakm$ be a $G$-toral subalgebra of $\frakg$: say, $\frakm\subseteq \Lie(S)$, where $S$ is a torus of $G$.  Then $N_G(\frakm)$ is reductive.  To see this, note that $N_G(\frakm)$ contains $T$, where $T$ is a maximal torus of $G$ containing $S$, so $N_G(\frakm)^0$ is generated by $T$ and by certain root groups $U_\alpha$ with respect to $T$.  But if $\alpha$ is a root, $g\in U_\alpha$ and $x\in \Lie(S)$ then $g\cdot x- x\in \Lie(U_\alpha)$, so $U_\alpha$ normalises $\frakm$ if and only if $U_\alpha$ centralises $\frakm$ if and only if $d\alpha$ annihilates $\frakm$.  We deduce that $U_\alpha\subseteq N_G(\frakm)$ if and only if $U_{-\alpha}\subseteq N_G(\frakm)$, and reductivity of $N_G(\frakm)$ follows.  This argument also shows that $N_G(\frakm)^0= C_G(\frakm)^0$.

Suppose further that $p$ is fabulous for $G$.  Then centralisers and normalisers are smooth, so
\begin{equation}
\label{eqn:toral_centraliser}
 \frakm\subseteq \frakn_\frakg(\frakm)= \frakc_\frakg(\frakm)= \Lie(C_G(\frakm)^0).
\end{equation}
Moreover, let $K= C_G(\frakm)^0$, let $Z= C_G(K)^0$ and let $L= C_G(Z)$. As $p$ is fabulous for $G$, $Z$ is smooth.  We have $Z\subseteq Z(K)^0$: for if $g\in C_G(K)$ then $g$ centralises $\frakm$ since $\frakm\subseteq \Lie(K)$ by (\ref{eqn:toral_centraliser}), so $C_G(K)^0$ is a connected subgroup of $C_G(\frakm)$.  This implies that $Z$ is a torus, as $K$ is reductive.  Hence $L$ is a Levi subgroup of $G$.  We have $\frakm\subseteq \Lie(Z)$ by smoothness of $Z$ since $K$ centralises $\frakm$.  By construction, $K\subseteq L$ and $\frakn_\frakg(\frakm)= \Lie(K)\subseteq \Lie(L)$.  Note that if $\frakm\not\subseteq \frakz(\frakg)$ then $L$ is proper.
\end{rem}

\begin{lem}
\label{lem:semisimple_toral}
 Let $\frakh$ be a subalgebra of $\frakg$ such that every element of $\frakh$ is semisimple.  Then $\frakh$ is $G$-toral.
\end{lem}

\begin{proof}
 Note that $\frakh$ is abelian by Remark~\ref{rem:toral}.
% First assume $\frakh$ is abelian.
 We use induction on $\dim(G)$.  The result holds trivially if $\dim(G)= 0$, so let $G$ be arbitrary.  If $\frakh\subseteq \frakz(\frakg)$ then $\frakh\subseteq \Lie(T)$ for any maximal torus $T$ of $G$, so we are done.  Otherwise there exists $x\in \frakh$ such that $C_G(x)^0$ is a proper reductive subgroup of $G$.  Now $C_G(x)^0$ is smooth by \cite[9.1\ Prop.]{borel} since $x$ is semisimple, so $\frakh\subseteq \Lie(C_G(x)^0)$.  By our induction hypothesis, $\frakh$ is $C_G(x)^0$-toral, so $\frakh$ is $G$-toral.
\end{proof}

\begin{lem}
\label{lem:toral_augmentation}
 Let $\frakh$ be a $G$-toral subalgebra of $\frakg$ and let $x\in \frakg$ such that $x$ is semisimple.  Suppose that $x$ centralises $\frakh$, or that $\char(k)= 0$ and $x$ normalises $\frakh$.  Then $k\cdot x+ \frakh$ is $G$-toral.
\end{lem}

\begin{proof}
 Choose an embedding of $G$ in $\GL_n$ for some $n\in \NN$.  If $x$ centralises $\frakh$ then $k\cdot x\cup \frakh$ consists of pairwise commuting semisimple matrices.  Hence by a standard theorem from linear algebra, the matrices in $k\cdot x\cup \frakh$ are simultaneously diagonalisable in $M_n(k)$.  This implies that every element of $k\cdot x+ \frakh$ is semisimple, so $k\cdot x+ \frakh$ is $G$-toral by Lemma~\ref{lem:semisimple_toral}.
 
 Now suppose $\char(k)= 0$ and $x$ normalises $\frakh$.  Since $x$ is semisimple, $x$ acts semisimply on $\frakh$.  Let $y\in \frakh$ be any nonzero eigenvector of $\ad(x)$, with eigenvalue $a$.  A simple calculation shows that $\ad(x)(y^t)= tay^t$ for any $t\in \NN$, where $y^t$ denotes the usual matrix power of $y$ in $M_n(k)= \Lie(\GL_n)$.   
 But $y^t\neq 0$ for any $t\in \NN$ as $y$ is semisimple and nonzero, which forces $ta= 0$ for all but finitely many $t$.  This implies that $a= 0$ since we are in characteristic 0.  We deduce that $x$ centralises $\frakh$, and the result follows from the previous paragraph.
\end{proof}

\begin{lem}
\label{lem:abelian}
 Suppose $p$ is fabulous for $G$.  Let $\frakh$ be an abelian $G$-completely reducible subalgebra of $\frakg$.  Then $\frakh$ is $G$-toral.
\end{lem}

\begin{proof}
 Let $x\in \frakh$.  Then $\frakc_\frakg(x)= \Lie(C_G(x))$, so if $y\in \frakc_\frakg(x)$ then $y_s\in \frakc_\frakg(x)$ and $y_n\in \frakc_\frakg(x)$, so $y_s$ and $y_n$ centralise $x$.  By a similar argument, $y_s$ and $y_n$ centralise $x_s$ and $x_n$.  It follows from the construction described in Remark~\ref{rems:closure_invt}(a) that $\frakh^J$ is abelian.  Now $\frakh^J$ is $G$-cr by Remark~\ref{rems:closure_invt}(b), and it is enough to prove that $\frakh^J$ is $G$-toral.  Hence we can assume without loss that $\frakh$ is Jordan-closed.
 
 By Lemma~\ref{lem:semisimple_toral}, it suffices to show that $\frakh$ consists of semisimple elements.  Suppose not.  Then there exists $0\neq y\in \frakh$ such that $y$ is nilpotent.  The 1-dimensional subspace $\frakm$ spanned by $y$ is an ideal of $\frakh$, so $\frakm$ is $G$-cr by Theorem~\ref{thm:lienormal}.  But this is impossible by Example~\ref{ex:nilpt}.  This completes the proof.
\end{proof}

\begin{lem}
\label{lem:solvable}
 Suppose $p$ is fabulous for $G$.  Let $\frakh$ be a solvable $G$-completely reducible subalgebra of $\frakg$.  Then $\frakh$ is $G$-toral.
\end{lem}

\begin{proof}
 We use induction on $\dim(G)+ \dim(\frakh)$.  The ideal $[\frakh, \frakh]$ is $G$-cr by Theorem~\ref{thm:lienormal} and $[\frakh, \frakh]\subseteq [\frakg, \frakg]= \Lie([G,G])$.  Now $[\frakh, \frakh]$ is a proper subalgebra of $\frakh$, so $[\frakh, \frakh]$ is $G$-toral by our induction hypothesis.  If $[\frakh, \frakh]= 0$ then $\frakh$ is abelian, so $\frakh$ is $G$-toral by Lemma~\ref{lem:abelian}.  Otherwise let $\frakh_1$ be the last nonzero term of the derived series of $\frakh$.  Then $\frakh_1$ is $G$-cr by Theorem~\ref{thm:lienormal}(a), so $\frakh_1$ is $G$-toral by Lemma~\ref{lem:abelian} as $\frakh_1$ is abelian.  Now $\frakh_1\subseteq [\frakh, \frakh]$ is not contained in $\frakz(\frakg)$ as $[G, G]$ is semisimple (see Remarks~\ref{rems:fabulous}(iv)), so by Remark~\ref{rem:toral_normaliser}, $\frakh\subseteq N_\frakg(\frakh_1)\subseteq \Lie(L)$ for some proper Levi subgroup $L$ of $G$.  Then $\frakh$ is $L$-cr by Proposition~\ref{prop:ascent_descent}(ii) and $\dim(L)< \dim(G)$, so $\frakh$ is $L$-toral by our induction hypothesis.  Hence $\frakh$ is $G$-toral, as required.
\end{proof}

\noindent We can now give a classification result for maximal solvable subalgebras of $\frakg$ when $p$ is fabulous for $G$: compare \cite[Thm.\ D(b)]{HeSt}.

\begin{prop}
\label{prop:solvable_Borel}
 Suppose $p$ is fabulous for $G$.  Let $\frakh$ be a solvable subalgebra of $\frakg$.  Then $\frakh\subseteq \frakh^J\subseteq \Lie(B)$ for some Borel subgroup $B$ of $G$.  In particular, a maximal solvable subalgebra of $\frakg$ is the Lie algebra of some Borel subgroup.
\end{prop}

\begin{proof}
 Since $\Lie(B)$ is Jordan-closed for any Borel subgroup $B$ of $G$ (Remarks~\ref{rems:closure_invt}(b)), it suffices to prove that $\frakh\subseteq \Lie(B)$ for some Borel subgroup $B$.  Let $\lambda\in Y(G)$ such that $(P_\lambda, L_\lambda)$ yields a $k$-semisimplification of $\frakh$.  Then $\fraks:= c_\lambda(\frakh)\subseteq \frakl_\lambda$ is solvable and $G$-cr and $\frakh\subseteq \fraks+ \Lie(R_u(P_\lambda))$.  Now $\fraks$ is $L_\lambda$-cr by Proposition~\ref{prop:ascent_descent}(ii), so $\fraks$ is $L_\lambda$-toral by Lemma~\ref{lem:solvable}: say, $\fraks\subseteq \Lie(S)$ for some torus $S$ of $L_\lambda$.  It follows that $\frakh\subseteq \Lie(S)+ \Lie(R_u(P_\lambda))= \Lie(SR_u(P_\lambda))$.  But $SR_u(P_\lambda)$ is contained in a Borel subgroup of $G$, so the first assertion follows.  The second assertion is immediate.
\end{proof}

\begin{cor}
\label{cor:solv_decompn}
 Suppose $p$ is fabulous for $G$.  Let $\frakh$ be a Jordan-closed solvable subalgebra of $\frakg$ and let $\frakn$ be the set of nilpotent elements of $\frakh$.  Let $\lambda\in Y(G)$ such that $(P_\lambda, L_\lambda)$ yields a semisimplification of $\frakh$.  Then:
 \begin{itemize}
  \item[(a)] $\frakn= \frakh\cap \Lie(R_u(P_\lambda))$.  In particular, $\frakn$ is an ideal of $\frakh$.
  \item[(b)] There is a $G$-toral subalgebra $\fraks$ of $\frakh$ such that $\frakh= \fraks\oplus \frakn$.
 \end{itemize} 
\end{cor}

\begin{proof}
 (a). Since $c_\lambda(\frakh)$ is solvable and $G$-cr, $c_\lambda(\frakh)$ consists of semisimple elements by Lemma~\ref{lem:solvable}.  Hence $c_\lambda$ kills every nilpotent element of $\frakh$, so $\frakn\subseteq \Lie(R_u(P_\lambda))$.  But $\Lie(R_u(P_\lambda))$ consists of nilpotent elements, so $\frakn= \frakh\cap \Lie(R_u(P_\lambda))$.  This is an ideal of $\frakh$ because $\frakh\subseteq \frakp_\lambda$.
 
 (b). Let $\fraks$ be a maximal $G$-toral subalgebra of $\frakh$.  We claim that $c_\lambda(\fraks)= c_\lambda(\frakh)$.  Suppose not.  Now $\fraks$ is abelian and consists of ad-semisimple elements, so $\fraks$ acts completely reducibly on $\frakh$.  Let $\frakh_0= \frakc_\frakh(\fraks)$ be the trivial weight space for the action of $\fraks$ on $\frakh$; then, since the $\fraks$-weight spaces of $\frakh$ corresponding to non-zero weights consist of nilpotent elements and $c_\lambda$ annihilates all such, we have $c_\lambda(\frakh_0)= c_\lambda(\frakh)$.  Note that $\fraks\subseteq \frakh_0$ as $\fraks$ is abelian.
 
 By hypothesis, there exists $x\in \frakh_0$ such that $c_\lambda(x)\not\in c_\lambda(\fraks)$.  As $p$ is fabulous for $G$, $\frakc_\frakg(\fraks)= \Lie(C_G(\fraks))$ is Jordan-closed, by Remarks~\ref{rems:closure_invt}(b), and $\frakh$ is Jordan-closed by assumption, so $\frakh_0= \frakc_\frakh(\fraks)= \frakc_\frakg(\fraks)\cap \frakh$ is Jordan-closed.  Hence $x_s\in \frakh_0$ and $c_\lambda(x_s)= c_\lambda(x)\not\in c_\lambda(\fraks)$; in particular, $x_s\not\in \fraks$.  Since $x_s$ is semisimple and commutes with $\fraks$, the subalgebra $\fraks\oplus k\cdot x_s$ of $\frakh$ is toral by Lemma~\ref{lem:toral_augmentation}.  But this contradicts the choice of $\fraks$.  We deduce that $c_\lambda(\fraks)= c_\lambda(\frakh)$, which implies that $\frakh= \fraks\oplus \frakn$.
\end{proof}

Note that \cite[Ch.~VII, \S 3, Cor.~2]{Bou05}
yields a version of Corollary \ref{cor:solv_decompn} in characteristic $0$.  David Stewart (private communication) has suggested an alternative proof of Proposition~\ref{prop:solvable_Borel} under weaker assumptions on $p$, using \cite[Cor.\ 1.3]{premetstewart}.

\section{Characteristic 0}
\label{sec:char0}

We keep our assumption that $k$ is algebraically closed.  If $H$ is a $G$-cr subgroup of $G$ then $H$ is reductive, and the converse also holds if ${\rm char}(k)= 0$ \cite[Prop.~4.2]{serre2}.  This gives an intrinsic characterisation of $G$-cr subgroups in characteristic 0: a subgroup $H$ is $G$-cr if and only if it is reductive.  (Note that a reductive subgroup of $G$ need not be $G$-cr in positive characteristic: e.g.,~see \cite[Ex.~3.45]{BMR}.)  The situation for Lie algebras is not so clear-cut.

\begin{ex}
\label{ex:not_intrinsic}
 Consider the 1-dimensional Lie algebra $k$.  We may embed $k$ in $\fraksl_2$ as the subalgebra of traceless diagonal matrices, or as the subalgebra of strictly upper triangular matrices.  It is clear that in the first case, the image of $k$ is $\SL_2$-cr, while in the second it is not.
\end{ex}

\noindent This example shows that even in characteristic 0, we cannot determine whether a subalgebra $\frakh$ of $\frakg$ is $G$-cr by looking only at the intrinsic properties of $\frakh$.  We can trace the problem to the lack of a Jordan decomposition for $\frakh$: there is no intrinsic notion of semisimple and nilpotent elements.  We do have notions of ad-semisimple and ad-nilpotent elements, but they don't detect properties of $\frakz(\frakh)$.

There is an intrinsic characterisation for when a restricted subalgebra $\frakh$ of $\frakg$ is $G$-cr when the characteristic is positive and sufficiently large, using the notion of ``$p$-reductive'' subalgebras: see \cite[Cor.\ 10.3]{HeSt}.  In \emph{loc.\ cit.}\ one makes use of the extra structure arising from the $p$-power map on $\frakh$.  Below we give a counterpart to \cite[Cor.\ 10.3]{HeSt} by providing a characterisation of $G$-cr subalgebras of $\frakg$ in characteristic 0 (Theorem~\ref{thm:Lie_Gcr_crit}), and we give an explicit description of a semisimplification of an arbitrary subalgebra in characteristic 0.  Different methods from those of \emph{op.\ cit.}\ are required, as there is no notion of restricted structure here.

The content of Theorem~\ref{thm:Lie_Gcr_crit} in characteristic 0 follows quickly from work of Richardson --- see Remark~\ref{rem:rich} below.  We give an independent proof, however, as some of our results hold under the weaker hypothesis that $p$ is fabulous for $G$.

The next result is a Lie algebra counterpart of \cite[Thm.\ 3.46]{BMR}; the analogues to the separability conditions in \emph{loc.\ cit.} hold here because of the assumption on $p$.

\begin{prop}
\label{prop:ad_ss_Gcr}
 Suppose $p$ is fabulous for $G$.  Let $\frakh$ be a subalgebra of $\frakg$ such that $\frakh$ acts semisimply on $\frakg$.  Then $\frakh$ is $G$-completely reducible.
\end{prop}

\begin{proof}
 Suppose for a contradiction that $\frakh$ is not $G$-cr.  Choose ${\mathbf x}= (x_1,\ldots, x_m)\in \frakh^m$ for some $m\in \NN$ such that the $x_i$ span $\frakh$.  The orbit $G\cdot {\mathbf x}$ is not closed (Theorem~\ref{thm:mcninch1}), so there exists $\lambda\in Y(G)$ such that ${\mathbf x}':= \lim_{a\to 0} \lambda(a)\cdot {\mathbf x}$ exists and $G\cdot {\mathbf x}'$ is closed.  This implies that $\dim(G\cdot {\mathbf x})> \dim(G\cdot {\mathbf x}')$. Consequently, letting $\frakh'$ be the subalgebra generated by the components of ${\mathbf x}'$, we have  $\dim(C_G(\frakh'))> \dim(C_G(\frakh))$, so $$\dim(\frakc_\frakg(\frakh'))= \dim(\Lie(C_G(\frakh')))> \dim(\Lie(C_G(\frakh)))= \dim(\frakc_\frakg(\frakh)),$$
 where the equalities hold because $p$ is fabulous for $G$.
 
 Now since $\ad_\frakg(\frakh)$ is $\GL(\frakg)$-cr and 
 $\ad_\frakg({\mathbf x})$
 is a generating tuple of $\ad_\frakg(\frakh)$, 
 $\GL(\frakg)\cdot \ad_\frakg({\mathbf x})$ is closed, by Theorem~\ref{thm:mcninch1}.  Hence $\ad_\frakg({\mathbf x}')= \ad_\frakg\left(\lim_{a\to 0} \lambda(a)\cdot {\mathbf x}\right)= \lim_{a\to 0} \lambda(a)\cdot \ad_\frakg({\mathbf x})$ is $\GL(\frakg)$-conjugate to $\ad_\frakg({\mathbf x})$.  But $\frakc_\frakg(\frakh')$ and $\frakc_\frakg(\frakh)$ are precisely the subsets of $\frakg$ annihilated by $\ad_\frakg(\frakh')$ and $\ad_\frakg(\frakh)$, respectively, so $\frakc_\frakg(\frakh')$ and $\frakc_\frakg(\frakh)$ are $\GL(\frakg)$-conjugate, which implies they have the same dimension.  This gives a contradiction.  We conclude that $\frakh$ must be $G$-cr after all, so we are done.
\end{proof}

Before we proceed to Theorem~\ref{thm:Lie_Gcr_crit}, we need to recall some standard Lie algebra theory.  If $\frakh$ is a Lie algebra then we define $\rad(\frakh)$ to be the \emph{solvable radical} of $\frakh$: that is, the unique largest solvable ideal of $\frakh$.  We say that $\frakh$ is \emph{semisimple} if $\rad(\frakh)= 0$, \cite[\S 3.1]{Hum72}.  
If $\char(k)= 0$ then any finite-dimensional representation of a semisimple Lie algebra is completely reducible \cite[Thm.~6.3]{Hum72}, and any Lie algebra $\frakh$ has a \emph{Levi decomposition} $\frakh= \frakk\oplus \rad(\frakh)$, where $\frakk$ is a semisimple subalgebra of $\frakh$ (not necessarily an ideal) \cite[\S 6.8 Thm.~5]{BouChap1}.  In this case, if $\rad(\frakh)$ is $G$-toral then $\rad(\frakh)= \frakz(\frakh)$.  For if $x\in \frakh$ is semisimple then $x$ centralises $\rad(\frakh)$ by Lemma~\ref{lem:toral_augmentation}; but the set of semisimple elements of $\frakk$ is dense as $\frakk$ is semisimple, so $\frakk$ centralises $\rad(\frakh)$, so $\rad(\frakh)\subseteq \frakz(\frakh)$.  The reverse inclusion follows because $\frakz(\frakk)= 0$.

\begin{thm}
\label{thm:Lie_Gcr_crit}
 Let $\frakh$ be a subalgebra of $\frakg$.  Consider the following conditions:
 \begin{itemize}
  \item[(i)] $\frakh$ acts semisimply on $\frakg$.
  \item[(ii)] $\frakh$ is $G$-completely reducible.
  \item[(iii)] $\rad(\frakh)$ is $G$-toral.
 \end{itemize}
 If $p$ is fabulous for $G$ then (i) $\implies$ (ii) $\implies$ (iii).  If $\char(k)= 0$ then (i)--(iii) are equivalent.
\end{thm}

\begin{proof}
 Suppose $p$ is fabulous for $G$.  If (i) holds then (ii) holds by Proposition~\ref{prop:ad_ss_Gcr}, while if (ii) holds then (iii) holds by Theorem~\ref{thm:lienormal} and Lemma~\ref{lem:solvable}.  Now suppose $\char(k)= 0$.  To complete the proof, it's enough to show that (iii) implies (i).  So suppose $\rad(\frakh)$ is $G$-toral.  Write $\frakh= \frakk\oplus \rad(\frakh)$, where $\frakk$ is semisimple.  Since $\char(k)= 0$, $\frakk$ acts semisimply on $\frakg$.  Now $\rad(\frakh)$ consists of pairwise commuting $\ad$-semisimple elements, so  $\rad(\frakh)$ acts semisimply on $\frakg$.  But $\rad(\frakh)$ commutes with $\frakk$, so it follows easily that $\frakh$ acts semisimply on $\frakg$, as required.
\end{proof}

Let $f\colon G\to M$ be a homomorphism of connected reductive groups. We say that $f$ is \emph{non-degenerate} if $\ker(f)^0$ is a torus.
The following is the Lie algebra counterpart of \cite[Cor.~4.3]{serre2}, see also \cite[Lem.~2.12(ii)]{BMR}.

\begin{cor}
\label{cor:Gcr_funct}
 Suppose ${\rm char}(k)= 0$.  
 Let $f\colon G\to M$ be a homomorphism of connected reductive groups. 
 Let $\frakh$ be a subalgebra of $\frakg$.  If $\frakh$ is $G$-completely reducible, then $df(\frakh)$ is $M$-completely reducible.  Conversely, if $f$ is non-degenerate and $df(\frakh)$ is $M$-completely reducible then $\frakh$ is $G$-completely reducible.
\end{cor}

\begin{proof}
  Write $\frakh= \frakk\oplus \frakz$, where $\frakk$ is semisimple and $\frakz= \rad(\frakh)$.  Then $df(\frakk)$ is semisimple and $df(\frakz))$ is solvable, so $df(\frakh)= df(\frakk)\oplus df(\frakz)$.  It follows that $\rad(df(\frakh))= df(\frakz)$.  If $\frakh$ is $G$-cr then $\frakz$ is $G$-toral by Theorem~\ref{thm:Lie_Gcr_crit}, so there is a torus $S$ of $G$ such that $\frakz\subseteq \fraks$.  Then $df(\frakz)$ is contained in the Lie algebra of the torus $f(S)$, so $df(\frakz)$ is $M$-toral.  Hence $df(\frakh)$ is $M$-cr by Theorem~\ref{thm:Lie_Gcr_crit}.
  
  Conversely, suppose $f$ is non-degenerate and $df(\frakh)$ is $M$-cr.  Then $df(\frakz)= \rad(df(\frakh))$ is $M$-toral by Theorem~\ref{thm:Lie_Gcr_crit} and thus consists of semisimple elements.  Since $\ker(df)= \Lie(\ker(f))$ is central in $\frakg$ and also consists of semisimple elements, it follows that $\frakz$ consists of semisimple elements.  Hence $\frakz$ is $G$-toral (Lemma~\ref{lem:semisimple_toral}), so $\frakh$ is $G$-cr by Theorem~\ref{thm:Lie_Gcr_crit}.
\end{proof}

\begin{rem}
\label{rem:closed_funct}
 Applying Theorem~\ref{thm:mcninch1}, we can translate Corollary~\ref{cor:Gcr_funct} into more geometric language.  Let $m\in \NN$ and let $(x_1,\ldots, x_m)\in \frakg^m$.  It follows from Corollary~\ref{cor:Gcr_funct} applied to the subalgebra $\frakh$ generated by the $x_i$ that if $G\cdot (x_1,\ldots, x_m)$ is closed then $M\cdot df(x_1,\ldots, x_m)$ is closed, and that the converse also holds if $f$ is non-degenerate.
\end{rem}

\begin{cor}
\label{cor:Jordan-closed}
 Suppose ${\rm char}(k)= 0$.  Let $\frakh$ be a $G$-completely reducible subalgebra of $\frakg$.  Then $\frakh$ is Jordan-closed.
\end{cor}

\begin{proof}
 We can write $\frakh= \frakk\oplus \rad(\frakh)$ for some semisimple subalgebra $\frakk$ of $\frakg$.  Now $\rad(\frakh)$ is $G$-toral by Theorem~\ref{thm:Lie_Gcr_crit}, so $\rad(\frakh)$ is Jordan-closed and $\rad(\frakh)= \frakz(\frakh)$.  By \cite[Lem.\ 3.2]{rich}, $\frakk$ is algebraic, so $\frakk$ is Jordan-closed.   It now follows easily that $\frakh$ is Jordan-closed.
\end{proof}

\begin{rem}
\label{rem:no_characterisation}
 The equivalence of (ii) and (iii) in Theorem~\ref{thm:Lie_Gcr_crit} can fail in positive characteristic.  For example, let $p$, $G$ and $\frakh$ be as in Example~\ref{ex:PGL_2}.  We observed earlier that $\frakh$ is $G$-ir, but it is easy to check that every element of $\frakh$ is nilpotent.
 
Conversely, let $k$ be algebraically closed of characteristic $3$, let $M= \SL_2$, let $V$ be the natural module for $M$, and consider the action of $M$ on the third symmetric power $W:=S^3V$---this gives an embedding of $M$ inside $G:=\GL_4$ and gives rise to a faithful representation of the Lie algebra $\frakm$.
We claim that $W$ is not semisimple as an $\frakm$-module, so $\frakm$ is not $G$-cr even though is has trivial radical.
The four-dimensional module $W$ has a basis $x^3,x^2y,xy^2,y^3$, where $x$ and $y$ can be identified with the standard basis vectors for $V$.
The basis vectors $x^3$ and $y^3$ are killed by the action of $\frakm$, so span a two-dimensional submodule with a trivial $\frakm$-action. 
The quotient by this submodule is a copy of the natural module $V$, which is simple, but there is no submodule of $W$ isomorphic as an $\frakm$-module to $V$, as direct calculation with the standard generators for $\frakm\cong \mathfrak{sl}_2$ will easily verify.

 Corollary~\ref{cor:Jordan-closed} can also fail in positive characteristic.  For, let $\char(k) = 2$ and let $G= \PGL_2(k)\times k^*\times k^*$ and let $\frakh$ be the subalgebra of $\frakg$ spanned by the elements of the form $\left(\ovl{\twobytwo{0}{a}{0}{0}}, a, 0\right)$ and $\left(\ovl{\twobytwo{0}{0}{b}{0}}, 0, b\right)$ for $a,b\in k$.  Clearly $\frakh$ is $G$-cr but is not Jordan-closed.
\end{rem}

\begin{rem}
\label{rem:rich}
 Richardson's seminal paper \cite{rich} laid the foundations for the study via GIT of $G$-complete reducibility for subgroups and subalgebras.  We explain how to obtain Theorem~\ref{thm:Lie_Gcr_crit} in the characteristic 0 case from Richardson's results.  Suppose $\char(k)= 0$.  If $\frakh$ is a subalgebra of $\frakg$ then there is a unique smallest subalgebra $\frakh^{\rm alg}$ of $\frakg$ such that $\frakh\subseteq \frakh^{\rm alg}$ and $\frakh^{\rm alg}$ is algebraic, and there is a unique connected subgroup $A(\frakh)$ of $G$ such that $\frakh^{\rm alg}= \Lie(A(\frakh))$.  In fact, the map $K\mapsto \Lie(K)$ gives an inclusion-preserving bijection from the set of connected subgroups of $G$ to the set of algebraic subalgebras of $\frakg$ \cite[13.1 Thm.]{humphreys}.
 
 Now let $m\in \NN$, let $x_1,\ldots, x_m\in \frakg$, let $\frakh$ be the subalgebra of $\frakg$ generated by the $x_i$ and let $M= A(\frakh)$.  Richardson showed that the following are equivalent (see \cite[Lem.\ 3.5 and Thm.\ 3.6]{rich}): (a) $G\cdot (x_1,\ldots, x_m)$ is closed; (b) $M$ is reductive; (c) $M$ acts semisimply on $\frakg$.  These conditions are equivalent to $\frakh$ being $G$-cr, by Theorem~\ref{thm:mcninch1}.  Note that if $M$ is reductive then $\rad(\frakh^{\rm alg})= \rad(\Lie(M))= \Lie(Z(M)^0)$ is $G$-toral, and it is not hard to see that the converse holds.  Moreover, it also straightforward to show that $\rad(\frakh^{\rm alg})= \rad(\frakh)^{\rm alg}$, and that $\rad(\frakh)^{\rm alg}$ is $G$-toral if and only if $\rad(\frakh)$ is $G$-toral.  The characteristic 0 case of Theorem~\ref{thm:Lie_Gcr_crit} now follows.
 \end{rem}

\begin{ex}
\label{ex:char0_ss}
 Assume $\char(k)= 0$.  We now give an explicit description of the $k$-semi-simplification of a subalgebra $\frakh$ of $\frakg$.  Write $\frakh= \frakk\oplus \frakm$, where $\frakk$ is semisimple and $\frakm= \rad(\frakh)$. 
 Then $\frakk$ is Jordan-closed (see the proof of Corollary~\ref{cor:Jordan-closed}), and it follows easily that $\frakh^J= \frakk\oplus \frakm^J$.  Now $\frakm^J$ is solvable by Proposition~\ref{prop:solvable_Borel}.  Write $\frakm^J= \fraks\oplus \frakn$ as in Corollary~\ref{cor:solv_decompn}.  We claim that $\frakk\oplus \fraks$ is a semisimplification of both $\frakh$ and $\frakh^J$.
 
 To see this, choose a parabolic subgroup $P$ of $G$ such that $P$ yields a semisimplification of $\frakh^J$.  Since $\frakk\oplus \fraks$ is $G$-cr by Theorem~\ref{thm:Lie_Gcr_crit}, we can choose $\lambda\in Y(G)$ such that $P_\lambda= P$ and $\lambda$ centralises $\frakk\oplus \fraks$.  Now $\frakm^J$ is an ideal of $\frakh^J$, so $(P_\lambda, L_\lambda)$ also yields a semisimplification of $\frakm^J$ by Theorem~\ref{thm:lienormal}(b).  Hence $c_\lambda$ kills $\frakn$ by Corollary~\ref{cor:solv_decompn}(a).  It follows that $c_\lambda(\frakh^J)= \frakk\oplus \fraks$.
 
 If $a\in k^*$ then $\lambda(a)\cdot \frakm\subseteq \frakp_\lambda$ and $c_\lambda(\lambda(a)\cdot \frakm)= c_\lambda(\frakm)$ since $\lambda$ centralises $\fraks$.  Hence $c_\lambda(\frakm)= c_\lambda(\frakm')$, where $\frakm'$ is the subspace of $\frakp_\lambda$ spanned by the subspaces $\lambda(a)\cdot \frakm$ for $a\in k^*$.  It is clear that $\frakm'$ contains $\frakm^J$, since $\frakn$ is contained in the sum of the nonzero weight spaces for $\lambda$ acting on $\frakp_\lambda$.  Hence $c_\lambda(\frakm)= c_\lambda(\frakm')= c_\lambda(\frakm^J)= \fraks$, and we conclude that $c_\lambda(\frakh)= \frakk\oplus \fraks$.  But $\frakk\oplus \fraks$ is $G$-cr, so the rest of the claim follows.
\end{ex}

%%%%%%%%%%%%%%%%%%%%%%%%%%%%%%%%%%%%%%%%%%%%%%%%%%%%%%%%%%%%%%%%%%%%%%
%%%%%%%%%%%%% Acknowledgements
%%%%%%%%%%%%%%%%%%%%%%%%%%%%%%%%%%%%%%%%%%%%%%%%%%%%%%%%%%%%%%%%%%%%%%

\bigskip

\noindent {\bf Acknowledgments}: The research of this work was supported in part by 
the DFG (Grant \#RO 1072/22-1 (project number: 498503969) to G.~R\"ohrle). 
We thank the referees for their careful reading of the paper and their helpful comments clarifying some points.  We are especially grateful to one referee for providing us with Proposition~\ref{prop:finite}, which is a substantial improvement on a result from the original version.

%\bigskip
%%%%%%%%%%%%%%%%%%%%%%%%%%%%%%%%%%%%%%%%%%%%%%%%%%%%%%%%%%%%%%%%%%%%%%
%%%%%%%%%%%%% bibliography
%%%%%%%%%%%%%%%%%%%%%%%%%%%%%%%%%%%%%%%%%%%%%%%%%%%%%%%%%%%%%%%%%%%%%%

\end{document}